\documentclass[twoside,a4paper,11pt,reqno]{amsart}
\usepackage{amsmath,amssymb,amsthm,amscd}
\usepackage[]{latexsym,amssymb,amsmath,amsfonts, amsthm, verbatim}
\usepackage[all,cmtip,ps]{xy}
\usepackage[latin1]{inputenc}
\usepackage[english]{babel}
\usepackage[usenames,dvipsnames]{xcolor}

\usepackage{graphicx}

\newcommand{\add}{\mathrm{add}}

\newcommand{\dbmod}{D^b(\text{mod})}

\newcommand{\Z}{\mathbb{Z}}

\newcommand{\N}{\mathbb{N}}

\DeclareMathOperator{\Hom}{Hom} \DeclareMathOperator{\End}{End}
\DeclareMathOperator{\Ext}{Ext} \DeclareMathOperator{\Tor}{Tor}
\DeclareMathOperator{\Ker}{Ker} \DeclareMathOperator{\Img}{Im}
\DeclareMathOperator{\im}{Im}

 \DeclareMathOperator{\pd}{proj.
dim} 

\DeclareMathOperator{\gldim}{gl. dim}

 \DeclareMathOperator{\tr}{tr}

\newcommand{\Acal}{\ensuremath{\mathcal{A}}}

\newcommand{\Dcal}{\ensuremath{\mathcal{D}}}
\newcommand{\Scal}{\ensuremath{\mathcal{S}}}

\newcommand{\Ccal}{\ensuremath{\mathcal{C}}}

\newcommand{\Bcal}{\ensuremath{\mathcal{B}}}

\newcommand{\cc}{\Ccal}
\newcommand{\ca}{\Acal}
\newcommand{\cb}{\Bcal}
\newcommand{\cs}{\Scal}
\newcommand{\cd}{\Dcal}
\newcommand{\ck}{\ensuremath{\mathcal{K}}}
\newcommand{\cl}{\ensuremath{\mathcal{L}}}
\newcommand{\cx}{\ensuremath{\mathcal{X}}}
\newcommand{\cy}{\ensuremath{\mathcal{Y}}}
\newcommand{\cz}{\ensuremath{\mathcal{Z}}}
\newcommand{\Mod}{\mathrm{Mod}}
\renewcommand{\mod}{\mathrm{mod}}

\newcommand{\Tria}{\mathrm{Tria\,}}
\newcommand{\tria}{\mathrm{tria\,}}

\newcommand{\Cone}{\mathrm{Cone}}
\newcommand{\Proj}{\mathrm{Proj}}
\newcommand{\proj}{\mathrm{proj}}

\newcommand{\acyc}{\mathrm{acyc}}

\theoremstyle{plain}
\newtheorem{thm}{Theorem}[section]
\newtheorem{prop}[thm]{Proposition}
\newtheorem{proposition}[thm]{Proposition}

\newtheorem{lemma}[thm]{Lemma}
\newtheorem{cor}[thm]{Corollary}
\newtheorem{corollary}[thm]{Corollary}
\newtheorem{ex}[thm]{Example}
\newtheorem*{ex*}{Example}

\theoremstyle{definition}

\newtheorem{remark}[thm]{Remark}
\theoremstyle{remark}
\newtheorem*{rem}{Remark}

\newcommand{\ra}{\rightarrow}

\newcommand{\ten}{\otimes}
\newcommand{\lten}{\overset{\mathbf{L}}{\ten}}
\newcommand{\rhom}{\mathbf{R}\mathrm{Hom}}
\newcommand{\RHom}{\mathbf{R}\mathrm{Hom}}

\newcommand{\mcx}{\mathcal{X}}

\newcommand{\confer}{\emph{cf.}}

\numberwithin{equation}{section}

\begin{document}

\renewcommand{\thefootnote}{\fnsymbol{footnote}}

\begin{center}

{\bf Ladders and simplicity of derived module categories}

\bigskip

{\sc Lidia Angeleri  H\" ugel\footnote{LAH acknowledges partial
support by the Unversity of Padova through Project CPDA105885/10
"Differential graded categories", by DGI MICIIN
MTM2011-28992-C02-01, and by the Comissionat Per Universitats i
Recerca de la Generalitat de Catalunya through Project 2009 SGR
1389.}, Steffen Koenig, Qunhua Liu\footnote{QL acknowledges partial support by Natural Science Foundation of China 11301272 and of Jiangsu Province BK20130899.}, Dong Yang\footnote{DY
acknowledges support by the DFG priority programme SPP 1388
YA297/1-1 and KO1281/9-1.}}
\bigskip
\end{center}

\address{Lidia Angeleri  H\" ugel
\\ Dipartimento di Informatica - Settore Matematica
\\ Universit\`a degli Studi di Verona
\\ Strada Le Grazie 15 - Ca' Vignal 2
\\ I - 37134 Verona, Italy}
\email{lidia.angeleri@univr.it}

\address{Steffen Koenig\\
Institute of Algebra and Number Theory,
University of Stuttgart \\ Pfaffenwaldring 57 \\ 70569 Stuttgart,
Germany} \email{skoenig@mathematik.uni-stuttgart.de}

\address{Qunhua Liu \\ Institute of Mathematics, School of Mathematical Sciences, Nanjing Normal University \\  Nanjing 210023,
P.R.China
\newline
\indent Institute of Algebra and Number Theory,
University of Stuttgart \\ Pfaffenwaldring 57 \\ 70569 Stuttgart,
Germany}
\email{05402@njnu.edu.cn}

\address{Dong Yang \\ Department of Mathematics, Nanjing University \\ Nanjing 210093, P. R. China
\newline
\indent Institute of Algebra and Number Theory,
University of Stuttgart \\ Pfaffenwaldring 57 \\ 70569 Stuttgart,
Germany}\email{yangdong@nju.edu.cn}

\date{\today}

\begin{quote}
{\footnotesize {\bf Abstract.} Recollements of derived module categories are investigated, using a new
technique, ladders of recollements, which are maximal mutation sequences.
The position in the ladder is shown to control whether a recollement
restricts from unbounded to another level of derived category. Ladders
also turn out to control derived simplicity on all levels. An algebra is
derived simple if its derived category cannot be deconstructed, that is,
if it is not the middle term of a non-trivial recollement whose outer
terms are again derived categories of algebras. Derived simplicity on
each level is characterised in terms of heights of ladders.

These results are complemented by providing new classes of examples of
derived simple rings, in particular indecomposable commutative rings, as
well as by a finite dimensionsal counterexample to the Jordan-H\"older
problem for derived module categories. Moreover, recollements are used to
compute homological and K-theoretic invariants.}
\end{quote}

\smallskip

\vspace{1ex}

\begin{quote}
\end{quote}

\vspace{2ex}

\tableofcontents

\section{Introduction}

Derived categories, introduced by Grothendieck and Verdier, have been
playing an increasingly important role in various areas of mathematics,
including representation theory, algebraic geometry, microlocal analysis and
mathematical physics. Major topics of current interest are substructures of
derived categories, such as bounded $t$-structures, which form the
`skeleton'
of Bridgeland's stability manifold, as well as comparisons of
derived categories. One way to compare derived categories and their
invariants, such as Grothendieck groups and Hochschild cohomologies, is by
derived equivalences, for instance by tilting. Another way has been introduced
by Beilinson, Bernstein and Deligne, who defined the concept of
a recollement \cite{BeilinsonBernsteinDeligne82} of a triangulated category.
A recollement (`gluing')
of a derived category by another two derived categories is a
diagram of six functors between these categories, generalising Grothendieck's
six functors. Beilinson, Bernstein and Deligne used recollements to define
the category of perverse sheaves over a stratified topological space as the
heart of a $t$-structure that is obtained by `gluing' standard $t$-structures
on the strata.

Viewing a recollement as a short exact sequence of categories, deconstructing
the middle term into smaller and possibly less complicated outer terms,
suggests to use recollements for computing homological invariants inductively
along sequences of recollements. The induction start then has to be based on
investigating `derived simple' algebras that cannot be deconstructed further.
This raises a number of questions to which we are going to provide (positive
or negative) answers.
\begin{itemize}
\item[(1)]
{ Which invariants can be computed inductively along recollements?}
\item[(2)]
{ Does the concept of derived simplicity depend on the choice of (unbounded,
bounded, ...) derived category?}
\item[(3)]
{ When do recollements lift or restrict between different levels of derived
categories?}
\item[(4)]
{ Which algebras are derived simple?}
\item[(5)]
{ Which algebras satisfy a derived version of the Jordan--H\"older
theorem?}
\end{itemize}

\smallskip

With respect to inductively computing invariants along recollements, Happel
showed in~\cite{Happel93} that conjectures on homological dimensions, such
as the finitistic dimension conjecture,
can be transfered to smaller algebras. It is known that a recollement of derived module categories
induces long exact sequences on algebraic K-theory \cite{ThomasonTrobaugh90,Yao92,NeemanRanicki04,Schlichting06}, on Hochschild homologies and cyclic homologies~\cite{Keller98}, and on Hochschild cohomologies~\cite{KoenigNagase09} and \cite{Han11}.
In~\cite{LiuVitoria11}, the third named author and Vit\'oria showed that for a
finite-dimensional piecewise hereditary algebra  over a
 field any bounded $t$-structure can be
obtained by `gluing' via recollements from standard $t$-structures on
the derived category of vector spaces.

Complementing and extending results in the literature, we  show
the following results in Proposition \ref{p:finiteness-of-global-dimension},
Lemma \ref{l:finite-sides} (b) and Proposition \ref{p:rank-and-recollement}.
\smallskip

\noindent{\bf Theorem I.} \label{t:main-1}
{\it
Let $A,B,C$ be rings. 
Suppose that $\cd(\Mod A)$ admits a recollement by $\cd(\Mod B)$ and $\cd(\Mod C)$.
\begin{itemize}
\item[(a)] The global dimension of $A$ is finite if and only if those of $B$ and $C$ are finite.
\item[(b)] If $A$ is finite-dimensional over a field, then so are $B$ and $C$ over the same field. Moreover, the ranks of the Grothendieck groups of $\mod B$ and $\mod C$ sum up to that of $\mod A$.
\end{itemize}
}

The following result on higher K-groups,
 Theorem \ref{t:splitting-K-groups}, is motivated by the work of Chen
and Xi \cite{ChenXi12a}, which also has been extended by these authors in parallel and independent work \cite{ChenXi12b}.

\smallskip

\noindent{\bf Theorem II.} \label{t:main-1-K}
{\it
Let $A,B,C$ be finite-dimensional algebras over a field.
Suppose that $\cd^-(\Mod A)$ admits a recollement by $\cd^-(\Mod B)$ and $\cd^-(\Mod C)$. Then we have an isomorphism of K-groups $\mathbb{K}_*(A)\cong \mathbb{K}_*(B)\oplus \mathbb{K}_*(C)$.
}

\smallskip

Concerning the question (2), Example
\ref{ex:Dminussimple-neq-Dsimple} provides a finite-dimensional algebra that
is $\cd^-(\Mod)$-simple but not $\cd(\Mod)$-simple, that is, it has  non-trivial recollements at $\cd(\Mod)$-level, but not at $\cd^-(\Mod)$-level, and Example~\ref{e:three-recollements-and-derived-simple} provides a finite-dimensional algebra that is $K^b(\proj)$-simple but not $\cd^-(\Mod)$-simple. So, the concept of being
derived simple strongly depends on the choice of derived categories. We
will clarify the connection between the different choices by characterising
derived simplicity on each level in terms of recollements on the unbounded level.
Our main tool here is the concept of ladders, which are collections of
`adjacent' recollements, whose number is called `height' of
the ladder (see Section~\ref{s:ladder}).
The following result combines Theorems \ref{t:dminus-simple-and-ladder},
\ref{t:kbproj-simplicity-and-ladder} and \ref{t:dbmod-simplicity-and-ladder}.
\smallskip

\noindent{\bf Theorem III.}\label{t:main-2}
{\it Let $A$ be a finite-dimensional algebra over a field. Then
\begin{itemize}
\item[(a)] $A$ is $\cd(\Mod)$-simple if and only if all non-trivial ladders of $\cd(\Mod A)$ have height $0$.
\item[(b)] $A$ is $\cd^-(\Mod)$-simple if and only if all non-trivial ladders of $\cd(\Mod A)$ have height $\leq 1$.
\item[(c)] $A$ is $\cd^b(\Mod)$-simple if and only if $A$ is $K^b(\proj)$-simple if and only if $A$ is $\cd^b(\mod)$-simple if and only if all non-trivial ladders of $\cd(\Mod A)$ have height $\leq 2$.
\end{itemize}
}

\smallskip

The proof of this theorem relies on our answer to the question (3) above. While lifting is easily seen not to be problematic, restricting
recollements is in general not possible, see Example~\ref{e:three-recollements}. Therefore, we provide several criteria for restrictions,
see Section  \ref{s:lifting-and-restricting-recollements}. These criteria
are in terms of particular objects and in terms of two of the six functors occurring in the recollement.

\smallskip

A stratification is a sequence of recollements  deconstructing a derived module category into `simple factors'. These simple factors can be viewed as
`composition factors' of the given derived module category. The derived Jordan--H\"older theorem in Question (5) asks whether a finite stratification exists and whether any two stratifications have the same set of simple factors (up to equivalence).
The validity of the derived Jordan--H\"older theorem has been disproved by Chen
and Xi~\cite{CX1,CX2} for a certain family of
infinite-dimensional algebras, proved by the first three authors~\cite{AKL2,AKL3} for hereditary artin algebras and
finite-dimensional piecewise hereditary algebras,  and proved by the third and fourth authors~\cite{LY} for finite group algebras. 
In this paper, we give a finite-dimensional counterexample.

\smallskip

Concerning the problem of classifying derived simple algebras, classes
of such algebras were recently discovered in \cite{LY,LiuYang14,HanQin16,LiLiping13}.
Here we add another large class of examples by proving that indecomposable
commutative rings always are derived simple. In particular, this result
implies the validity of the Jordan--H\"older Theorem for derived module
categories over commutative noetherian rings.

\smallskip

Finally, we point out that Section~\ref{corrigendum} corrects a mistake in \cite{AKL3}.

\medskip

The paper is structured as follows. In Section~\ref{s:derived-categories-and-recollements} we collect some preliminaries on recollements of triangulated and derived categories, derived functors, derived simple algebras and stratifications. Moreover, we prove Theorem I(a) and the first statement of Theorem I(b). In Section~\ref{s:ladder} we recall the definition of a ladder and discuss when a ladder can be extended upwards or downwards. In Section~\ref{s:lifting-and-restricting-recollements} we study the lifting and restricting problem between different levels of derived categories; we give criteria when a recollement on $\cd(\Mod)$ can be restricted to a recollement on $\cd^-(\Mod)$, $\cd^b(\Mod)$, $K^b(\proj)$ and $\cd^b(\mod)$ respectively. Section~\ref{s:derived-simplicity} is devoted to studying the dependence of derived simplicity on the choice of (unbounded, bounded, ...) derived categories, in particular,  proving Theorem III.  Further, we show that indecomposable commutative algebras are derived simple for any choice of derived categories. As a consequence, we recover some results on silting and tilting over commutative algebras in the literature \cite{CM,PT,StovicekPospisil12}. In Section~\ref{s:k-theory} we study algebraic K-theory of finite-dimensional algebras and prove Theorem II. In the final Section~\ref{JH} we discuss the derived Jordan--H\"older theorem.

\medskip

\noindent{\it Acknowledgements:} We would like to thank Xiao-Wu Chen, Martin Kalck, Bernhard Keller,
Henning Krause, Changchang Xi and Guodong Zhou for helpful conversations, and in particular  we thank Changchang Xi for generously sharing his insight on algebraic K-theory and Yang Han for pointing out an error in an earlier version.

\section{Derived categories and recollements}\label{s:derived-categories-and-recollements}

In this section, the basic definitions are given and various
technical results are established that will be crucial in the later
sections. Subsection \ref{tria} contains some preliminaries on
triangulated categories, while  Subsection \ref{der}  focusses on
derived module categories. In Subsection~\ref{sss:left-and-right-adjoints} we study the problem when triangle functors of derived module categories restrict to subcategories.  In  \ref{ss:recollement-of-der-cat} we
collect techniques to investigate  functors and objects  involved in
recollements  of derived module categories, and in \ref{reminder} we
recall some important constructions of such recollements. The notions of  derived simple ring and stratification are reviewed in
\ref{dersimple} and \ref{strat}. {In \ref{sss:finiteness-of-global-dim} we show that given a recollement of derived module categories the middle algebra has finite global dimension if and only if so do the two outer algebras.}
\smallskip

Let $k$ be a commutative ring. We will view rings as k-algebras over a suitable commutative ring k. When $k$ is a field, let $D=\Hom_k(?,k)$ be the $k$-dual.

\subsection{Triangulated categories}\label{tria}

Let $\cc$ be a triangulated $k$-category with shift functor $[1]$.
An object $X$ of $\cc$ is \emph{exceptional} if $\Hom_{\cc}(X,
X[n])=0$ unless $n=0$. Let $\cs$ be a set of objects of $\cc$. We
denote by $\tria\cs$ the smallest triangulated subcategory of $\cc$
containing $\cs$ and closed under taking direct summands, and by
$\cs^\perp$ the \emph{right perpendicular category} of $\cs$, i.e.
\[\cs^\perp=\{X\in\cc|\Hom_\cc(Y, X[n])=0 \text{ for all }
Y\in\cs\text{ and all } n\in\mathbb{Z}\}.\] $\cs$ is called a set of
\emph{generators} of $\cc$ if $\cc=\tria\cs$.

Assume further that $\cc$ has all (set-indexed) infinite direct
sums. An object $X$ of $\cc$ is \emph{compact} if the functor
$\Hom_\cc(X,?)$ commutes with taking direct sums. For a set $\cs$ of
objects of $\cc$, we denote by $\Tria\cs$ the smallest triangulated
subcategory of $\cc$ containing $\cs$ and closed under taking direct
sums. $\cs$ is called a set of \emph{compact generators} of $\cc$ if
all objects in $\cs$ are compact and $\cc=\Tria\cs$. In this case
$\cc$ is said to be \emph{compactly generated} by $\cs$.

\subsubsection{\sc Recollements}\label{2.1.1}\label{sss:recollement}

A \emph{recollement}~\cite{BeilinsonBernsteinDeligne82} of
triangulated $k$-categories is a diagram
\begin{eqnarray}\xymatrix@!=3pc{\mathcal{C}'
\ar[r]|{i_*=i_!} &\mathcal{C} \ar@<+2.5ex>[l]|{i^!}
\ar@<-2.5ex>[l]|{i^*} \ar[r]|{j^!=j^*} &
\mathcal{C}''\ar@<+2.5ex>[l]|{j_*}
\ar@<-2.5ex>[l]|{j_!}}\label{eq:recollement}\end{eqnarray}of
triangulated categories and triangle functors such that
\begin{enumerate}
\item $(i^\ast,i_\ast)$,\,$(i_!,i^!)$,\,$(j_!,j^!)$ ,\,$(j^\ast,j_\ast)$
are adjoint pairs;

\item  $i_\ast,\,j_\ast,\,j_!$  are full embeddings;

\item  $i^!\circ j_\ast=0$ (and thus also $j^!\circ i_!=0$ and
$i^\ast\circ j_!=0$);

\item  for each $C\in \mathcal{C}$ there are triangles
\[
\xymatrix@R=0.5pc{
i_! i^!(C)\ar[r]& C\ar[r]& j_\ast j^\ast (C)\ar[r]& i_!i^!(C)[1]\\
j_! j^! (C)\ar[r]& C\ar[r]& i_\ast i^\ast(C)\ar[r]& j_!j^!(C)[1]
}
\]
where the maps are given by adjunctions.
\end{enumerate}
Thanks to (1) and (3), the two triangles (often called the \emph{canonical triangles}) in (4) are unique up to  unique isomorphisms.

We say that this is a recollement of $\cc$ by $\cc'$ and $\cc''$. We say the recollement is trivial if one of the triangulated
categories $\Ccal'$ and $\Ccal''$ is trivial, or equivalently, one
of the full embeddings $i_*$, $j_!$ and $j_*$ is a triangle
equivalence.

\subsubsection{\sc TTF triples}\label{2.1.2}\label{sss:TTF} An equivalent language for studying recollements is that of TTF triples. We are going to extend
\cite[4.2]{NZ}, which covers the case of compactly generated triangulated
categories.

Let $\cc$ be a triangulated $k$-category with shift functor $[1]$. A
\emph{$t$-structure} on $\cc$ is a pair of full subcategories
$(\cc^{\leq 0},\cc^{\geq 0})$ which are closed under isomorphisms and which satisfy the following conditions
\begin{itemize}
\item[--] $\cc^{\leq 0}[1]\subseteq \cc^{\leq 0}$ and $\cc^{\geq
0}[-1]\subseteq\cc^{\geq 0}$,
\item[--] $\Hom(X,Y)=0$ for $X\in\cc^{\leq 0}$ and $Y\in\cc^{\geq
0}[-1]$,
\item[--] for any object $X$ of $\cc$ there is a triangle
\[\xymatrix{X'\ar[r] &X\ar[r] & X''\ar[r] & X'[1]}\]
with $X'\in\cc^{\leq 0}$ and $X''\in\cc^{\geq 0}[-1]$.
\end{itemize}

A \emph{TTF triple} of $\cc$ is a triple $(\cx,\cy,\cz)$ of full
subcategories of $\cc$ such that $(\cx,\cy)$ and $(\cy,\cz)$ are
$t$-structures on $\cc$. It follows that $\cx$, $\cy$ and $\cz$ are
actually triangulated subcategories of $\cc$.
Given a recollement of the form (\ref{eq:recollement}), one easily
checks that $$(\Img(j_!), \Img(i_*), \Img(j_*))$$ is a TTF triple of
$\cc$, which we shall call the {\it associated TTF triple}.
We say that two recollements are equivalent if the TTF
triples associated to them coincide. The following result is
essentially included in~\cite[Section
1.4.4]{BeilinsonBernsteinDeligne82}, see also~\cite[Section
9.2]{Neeman99}, \cite[Section 4.2]{NZ}.

\begin{proposition}\label{p:recollement-corresponds-to-TTF}
Let $\cc$ be a triangulated $k$-category. There is a one-to-one
correspondence between equivalence classes of recollements of $\cc$
and TTF triples of $\cc$.
\end{proposition}

The key point of the proof is the following lemma.

\begin{lemma}\label{l:complete-ses-to-t-str} Let
$0\rightarrow\cc'\stackrel{F}{\rightarrow}\cc\stackrel{G}{\rightarrow}\cc''\rightarrow
0$ be a short exact sequence of triangulated $k$-categories, i.e.
$F$ is fully faithful and $G$ induces a triangle equivalence
$\cc/\Img F\cong\cc''$ (possibly up to direct summands). 
Then the following holds.
\begin{itemize}
\item[(a)] $F$ has a left adjoint if and only if so does $G$. In this case, if $G'$ denotes the left adjoint of $G$, then $(\Img G',\Img F)$ is a
$t$-structure on $\cc$.

\item[(b)] $F$ has a right adjoint if and only if so does $G$. In this case, if $G''$ denotes the right adjoint of $G$, then $(\Img F,\Img G'')$ is a
$t$-structure on $\cc$.
\end{itemize}
\end{lemma}

When $\Ccal$ is compactly generated by one object - for example,
$\Ccal$ is the derived module category of some (dg)-$k$-algebra -
a TTF triple is uniquely determined by its first term,
see \cite[4.4.14]{NZ}. In other words, the associated recollement is
uniquely determined by $\Img(j_!)$. Following \cite{NZ}, if
$\Ccal''$ is compactly generated by one object $T$, we say the {\it
recollement is generated by} $j_!(T)$.

\smallskip

The following result shows that recollements preserve direct sum
decompositions, which generalises a result implicitly appearing in
the proof of~\cite[Corollary 3.4]{LY}.

\begin{lemma}\label{l:commutating-direct-sum-decomposition-and-recollements}
Assume that a recollement of the form (\ref{eq:recollement}) is
given. Suppose that $\cc=\cc_1\oplus\ldots\oplus\cc_s$ is a direct
sum decomposition of triangulated categories. Then there are direct
sum decompositions $\cc'=\cc'_1\oplus\ldots\oplus\cc'_s$ and
$\cc''=\cc''_1\oplus\ldots\oplus\cc''_s$ such that the given
recollement restricts to recollements
$$\xymatrix@!=3pc{\mathcal{C}_i' \ar[r]|{i_*=i_!} &\mathcal{C}_i \ar@<+2.5ex>[l]|{i^!}
\ar@<-2.5ex>[l]|{i^*} \ar[r]|{j^!=j^*} &
\mathcal{C}_i''\ar@<+2.5ex>[l]|{j_*} \ar@<-2.5ex>[l]|{j_!}}$$ and
the direct sum of these restricted recollements is the given recollement.
\end{lemma}
\begin{proof} Let $(\cx,\cy,\cz)$ be the TTF triple corresponding to
the given recollement. For $i=1,\ldots,s$, let $\cx_i=\cx\cap\cc_i$,
$\cy_i=\cy\cap\cc_i$ and $\cz_i=\cz\cap\cc_i$. Then $\cx=\bigoplus_i
\cx_i$, $\cy=\bigoplus_i \cy_i$ and $\cz=\bigoplus_i \cz_i$, because
each object $C$ of $\cc$ is the direct sum $C=\bigoplus_i C_i$ of
objects $C_i\in\cc_i$ and $\cx$, $\cy$ and $\cz$ are closed under
taking direct summands. Moreover, each triangle in $\cc$ is the
direct sum of $s$ triangles, respectively lying in
$\cc_1,\ldots,\cc_s$. It follows that for any $i=1,\ldots,s$, the
triple $(\cx_i,\cy_i,\cz_i)$ is a TTF triple for $\cc_i$.
\end{proof}

\medskip

\subsection{Derived categories}\label{der}
Let  $A$ be a $k$-algebra. We denote by $\mathrm{Mod}A$ the category
of  (right) $A$-modules, by $\mathrm{mod}A$ its subcategory
consisting of modules with projective resolution by finitely
generated projectives, and by $\proj A$ its subcategory of finitely
generated projective modules. The analogous categories of left modules will
be denoted by $A\mbox{-}\Mod$, $A\mbox{-}\mod$, $A\mbox{-}\proj$ respectively. For
$*\in\{b,-,+,\emptyset\}$, let
$\cd^*(\mathrm{Mod}A)$ denote the derived category of complexes of
objects in  $\mathrm{Mod}A$
satisfying the corresponding boundedness condition.   {\it When $k$ is a field and $A$ is a finite-dimensional
$k$-algebra}, we also consider the corresponding derived categories $\cd^*(\mathrm{mod}A)$ of
objects in  $\mathrm{mod}A$. Notice that  $\cd^b(\mod A)$ then
 coincides with
 the full subcategory $\cd_{fl}(A)$ of $\cd(\Mod A)$ consisting of
complexes of $A$-modules whose total cohomology has finite length
over $k$.  The shift
functor will be denoted by $[1]$.

 Let $K^b(\proj A)$ denote the
homotopy category of bounded complexes of finitely generated
projective $A$-modules. We often view it as a full subcategory of
the categories $\cd^*(\mod A)$ and $\cd^*(\Mod A)$ and identify it
with its essential image.
The objects in $K^b(\proj A)$ are, up to isomorphism, precisely the
compact objects in $\cd(\Mod A)$. The free module $A_A$ of rank $1$
is a compact generator of $\cd(\Mod A)$.
 For a complex $X\in\cd(\Mod A)$  we write $X^{\tr_A}$ for $\rhom_A(X,A)$. When it does not cause confusion, we will drop the subscript $A$ and simply write $X^{\tr}$.

A complex $P$ in $K^b(\proj A)$ is said to have {\em length} $n$ if
$n$ is the minimal integer such that there is $Q\in K^b(\proj A)$ such that $Q\cong P$ and $n$
equals the number of non-zero components of $Q$.

If $k$ is a field and $A$ is finite-dimensional over $k$, then any object in  $\cd^-(\mod A)$
admits a \emph{minimal}
representative, that is, a complex of finitely generated projective
$A$-modules such that the images of the differentials lie in the
radicals.

\begin{lemma}\label{l:compact-and-finite-dimensional}
\begin{itemize}
\item[(a)]
The category $\cd_{fl}(A)$ is the full subcategory of $\cd(\Mod A)$
consisting of those objects $X$ such that the total cohomology of
the complex $\rhom_A(P,X)$ has finite length over $k$, i.e.
$\bigoplus_{n\in\mathbb{Z}}\Hom(P,X[n])$ has finite length, for any
$P\in K^b(\proj A)$.
\item[(b)]
The category $\cd^b(\Mod A)$ is the full subcategory of $\cd(\Mod
A)$ consisting of those objects $X$ such that the the complex
$\rhom_A(P,X)$ has bounded total cohomology, i.e. $\Hom(P,X[n])\neq
0$ for only finitely many $n\in\mathbb{Z}$, for any $P\in K^b(\proj
A)$.
\end{itemize}

Assume that $k$ is
a field and that $A$ is a finite-dimensional $k$-algebra.
\vspace{-5pt}
\begin{itemize}
\item[(c)]
The category $K^b(\proj A)$ is the full subcategory of $\cd(\Mod
A)$ consisting of those objects $P$ such that the complex
$\rhom_A(P,X)$ has finite-dimensional total cohomology for any
$X\in \cd^b(\mod A)$.
\item[(d)]
The category $K^b(A\mbox{-}\proj)$ is the full subcategory of $\cd(A\mbox{-}\Mod)$
consisting of those objects $P$ such that the complex $X\lten_A
P$ has finite-dimensional total cohomology for any $X\in
\cd^b(\mod A)$.
\end{itemize}
\end{lemma}

\begin{proof}
(a) It follows by d\'evissage that objects in $\cd_{fl}(A)$ satisfy
the condition. Since d\'evissage will be often used later, we give
the details here. Let $N$ be an object of $\cd_{fl}(A)$ and let $\ca$ be
the full subcategory of $\cd(\Mod A)$ consisting of objects $M$ such
that the total cohomology of $\rhom_A(M,N)$ has finite length over
$k$. Then $\ca$ contains $A$ and is closed under direct summands,
shifts and extensions. This shows that $\ca$ contains
$\tria(A)=K^b(\proj A)$ and we are done. Conversely taking $P=A$ we
see that the total cohomology space of $X=\rhom_A(A,X)$ in the
latter category has finite length, i.e. $X\in\cd_{fl}(A)$.

(b) Similar to (a).

(c) It follows by d\'evissage that objects of $K^b(\proj A)$ belong
to the latter category. Conversely, let $P$ be any object of the latter category.
Then $DP=\RHom_A(P,DA)$, and hence $P$, have finite-dimensional total cohomology, so $P\in\cd^b(\mod A)$. Thus we may
assume that $P$ is a minimal right bounded complex of finitely
generated projective $A$-modules. If such a complex is not bounded,
then some indecomposable projective $A$-module with simple top $S$ occurs infinitely
many times. It follows that there are nonzero morphisms from $P$
to infinitely many shifts of $S$, that is, the total cohomology of $\rhom_A(P,S)$ is infinite dimensional, a contradiction.

(d) Similar to (c).
\end{proof}

\smallskip

\subsubsection{\sc Restricting triangle functors}\label{sss:left-and-right-adjoints}
In this subsection we discuss restriction of triangle functors
between derived modules categories to subcategories.

Let $A$ and $B$ be $k$-algebras and $F:\cd(\Mod
A)\rightarrow\cd(\Mod B)$ a $k$-linear triangle functor. Let $\cd=K^b(\proj), \cd_{fl}, \cd^b(\mod), \cd^b(\Mod)$ or $\cd^-(\Mod)$. Let $\hat{\cd}$ be the essential image of $\cd$ under the canonical embedding into $\cd(\Mod)$. We say
that $F$ {\it restricts} to $\cd$ if $F$ restricts to a
triangle functor $\hat{\cd}_A\rightarrow\hat{\cd}_B$. The following two
well-known lemmas follow by d\'evissage.

\begin{lemma}\label{l:restriction-of-functors-to-compact-obj}
The following two conditions are equivalent:
\begin{itemize}
\item[(i)] $F$ restricts to $K^b(\proj)$,
\item[(ii)] $F(A)\in K^b(\proj B)$.
\end{itemize}
\end{lemma}

\begin{lemma}\label{l:restriction-of-functors-to-finite-dimensional-obj}
Assume that $k$ is a field and that $A$ and $B$ be
finite-dimensional $k$-algebras. The following two conditions are
equivalent:
\begin{itemize}
\item[(i)] $F$ restricts to $\cd^b(\mod)$,
\item[(ii)] $F(S)\in \cd^b(\mod B)$ for any simple $A$-module $S$.
\end{itemize}
\end{lemma}

\begin{lemma}\label{l:restriction-of-right-adjoint}
Assume that $F$ admits a right adjoint $G$. Consider the following
conditions
\begin{itemize}
\item[(i)] $F$ restricts to $K^b(\proj)$,
\item[(ii)] $G$ restricts to $\cd^b(\Mod)$,
\item[(iii)] $G$ restricts to $\cd_{fl}$.
\end{itemize}
Then (i) implies (ii) and (iii). If $k$ is a field and $A$ and $B$
are finite-dimensional $k$-algebras, then (iii) implies (i).
\end{lemma}
\begin{proof} (i)$\Rightarrow$(ii) Let $M$ be in $\cd^b(\Mod B)$
 and $P$ be any object in $K^b(\proj A)$. Then
\begin{eqnarray*}\Hom_A(P,G(M[n]))
=\Hom_B(F(P),M[n]).\end{eqnarray*} The condition (i) implies that
$F(P)\in K^b(\proj B)$. It follows from
Lemma~\ref{l:compact-and-finite-dimensional} (b) applied to the
algebra $B$ that the above space does not vanish for only finitely
many $n\in\mathbb{Z}$. Applying
Lemma~\ref{l:compact-and-finite-dimensional} (b) to the algebra $A$
shows that $G(M)\in\cd^b(\Mod A)$.

(i)$\Rightarrow$(iii) This follows from
Lemma~\ref{l:compact-and-finite-dimensional} (a) by a similar
argument as in (i)$\Rightarrow$(ii).

(iii)$\Rightarrow$(i) This follows from
Lemma~\ref{l:compact-and-finite-dimensional} (c) by a similar
argument as in (i)$\Rightarrow$(ii).
\end{proof}

Let $X$ be a complex of $A$-$B$-bimodules. Then there is a pair of
adjoint triangle functors
\[?\lten_A X:\cd(\Mod A)\rightarrow \cd(\Mod B),\qquad \rhom_B(X,?):\cd(\Mod B)\rightarrow \cd(\Mod A).\]
Assume that $A$ is projective as a $k$-module. Let $\mathbf{p}X$ be
a bimodule $\ck$-projective resolution of $X$
(see~\cite{Spaltenstein88}). Then $(\mathbf{p}X)_B$ is a
$\ck$-projective resolution of $X$ as a complex of right
$B$-modules, and hence $\rhom_B(X,?)$ is naturally isomorphic to
$\rhom_B(\mathbf{p}X,?)$, which is computed as the total complex of
the Hom bicomplex. In particular, for a complex $Y$ of
$C$-$B$-bimodules, $\rhom_B(\mathbf{p}X,Y)$ has the structure of a
complex of $C$-$A$-bimodules. By abuse of notation, we will identify
$\rhom_B(X,Y)$ and $\rhom_B(\mathbf{p}X,Y)$.

\begin{lemma}\label{l:existence-of-left-adjoint} Let $X$ be a
complex of $A$-$B$-bimodules and assume that $A$ is projective as a
$k$-module. Consider the following conditions
\begin{itemize}
\item[(i)] ${}_AX\in K^b(A\mbox{-}\proj)$,
\item[(ii)] $?\lten_A X$ has a left adjoint which restricts to
$K^b(\proj)$, in this case, the left adjoint is given by  $?\lten_B X^{\tr_{A^{op}}}$,
\item[(iii)] $?\lten_A X:\cd(\Mod A)\rightarrow\cd(\Mod B)$
restricts to $\cd_{fl}$.
\end{itemize}
Then
(i)$\Rightarrow$
(ii)$\Rightarrow$(iii). If $k$
is a field and $A$ and $B$ are finite-dimensional $k$-algebras, then
all three conditions are equivalent.
\end{lemma}

\begin{proof}

(i)$\Rightarrow$(ii): It follows by d\'evissage that
$(X^{\tr_{A^{op}}})_A=\RHom_{A^{op}}(X,A^{op})_A\in K^b(\proj A)$. Therefore, by \cite[Lemma 6.2 (a)]{Keller94}, the derived functor
\[
\RHom_A(X^{\tr_{A^{op}}},?):\cd(\Mod A)\rightarrow\cd(\Mod B)
\]
is isomorphic to $?\lten_A (X^{\tr_{A^{op}}})^{\tr_A}$, where $(X^{\tr_{A^{op}}})^{\tr_A}=\RHom_A(\RHom_{A^{op}}(X,A^{op}),A)$. Further,
the canonical morphism in $\cd(\Mod A^{op}\ten B)$
\[
X\longrightarrow (X^{\tr_{A^{op}}})^{\tr_A}
\]
is an isomorphism since restricting to $\cd(\Mod A^{op})$ it is an isomorphism for $X=A^{op}$.
It follows that $?\lten_A
X\cong\RHom_A(X^{\tr_{A^{op}}},?)$ and thus has a left adjoint $?\lten_B
X^{\tr_{A^{op}}}$, which restricts to $K^b(\proj B)\rightarrow
K^b(\proj A)$ by Lemma~\ref{l:restriction-of-functors-to-compact-obj}.

(ii)$\Rightarrow$(iii): This follows from
Lemma~\ref{l:restriction-of-right-adjoint}.

Finally, when $k$ is a field and $A$ and $B$ are finite-dimensional
$k$-algebras, the implication (iii)$\Rightarrow$(i) follows from
Lemma~\ref{l:compact-and-finite-dimensional} (d).
\end{proof}

\smallskip

\subsubsection{\sc Recollements of derived
categories}\label{ss:recollement-of-der-cat}

Suppose there is a recollement
\begin{align*}\label{eq:recollement-unbounded-der}\xymatrix@!=6pc{\cd(\Mod B) \ar[r]|{i_*=i_!} &\cd(\Mod A)
\ar@<+2.5ex>[l]|{i^!} \ar@<-2.5ex>[l]|{i^*} \ar[r]|{j^!=j^*} &
\cd(\Mod C)\ar@<+2.5ex>[l]|{j_*} \ar@<-2.5ex>[l]|{j_!}
}\tag{$R$}\end{align*} where $A$, $B$ and $C$ are $k$-algebras. By
\cite[5.2.9]{NZ}, \cite{NS1}, or \cite[Theorem 2.2]{AKL2}, the
recollement is generated
 by the compact exceptional object  $T=j_!(C)\in\cd(\Mod A)$.
 Moreover, the object $i^*(A)$ is a compact
generator of $\cd(\Mod B)$ by \cite[4.3.6, 4.4.8]{NZ}. The following
lemma is well-known.

\begin{lemma}\label{l:useful-info}
With the above notation,
 the four objects $T=j_!(C)$, $j_!j^!(A)$, $i_*i^*(A)$ and
$T'=i_*(B)$ of $\cd(\Mod A)$ satisfy
\begin{itemize}
\item[(a)] $j_!(C)$ and $j_!j^!(A)$ are left orthogonal to
both $i_*i^*(A)$ and $i_*(B)$, and $j_!j^!(A)\in\Tria(j_!(C))$;
\item[(b)] $j_!(C)$ and $i_*(B)$ are exceptional objects with
 $C\cong\End_A(T)$ and $B\cong\End_A(T')$;
\item[(c)] $\tria i_*i^*(A)=\tria i_*(B)$;
\item[(d)] $i_*i^*(A)\cong\rhom_A(i_*i^*(A),i_*i^*(A))$
as complexes of $k$-modules;
\item[(e)] $i^*$ and $j_!$ restrict to $K^b({\rm proj})$, and $i_*$ and
$j^!$ restrict to $\mathcal D^b({\rm Mod})$ as well as to $\mathcal D_{fl}$;
\item[(f)]
 $j_!j^!(A)$,
$i_*i^*(A)$ and $i_*(B)$ belong to $\cd^b(\Mod A)$.
\end{itemize}
\end{lemma}
\begin{proof}
(a) and (b) follows immediately from the definition of recollement.

(c) Since $i^*(A)$ is a compact generator of $\cd(\Mod B)$, it
follows that $\tria i^*(A)=\tria B$. Thus $\tria i_*i^*(A)=\tria
i_*(B)$ for $i_*$ is a full embedding, cf.~\cite[Lemma 2.2]{Neeman92a}.

(d) This is obtained by applying $\rhom_A(?,i_*i^*(A))$ to the
canonical triangle
\[
\xymatrix{
j_!j^!(A)\ar[r] & A\ar[r]& i_*i^*(A)\ar[r] &
j_!j^!(A)[1].
}
\]

(e) The statement on $i^*$ and $j_!$ follows from Lemma~\ref{l:restriction-of-functors-to-compact-obj} since $T$ and $T'$ are compact. For the second statement  apply Lemma~\ref{l:restriction-of-right-adjoint} on the right adjoints  $i_*$ and
$j^!$.

(f) It follows from the above canonical triangle and (c) that $j_!j^!(A)$, $i_*i^*(A)$ and $i_*(B)$
all belong to $\cd^b(\Mod A)$ if and only if so does one of them.
Since  $i_*(B)$ does by (e), the claim is proven.
\end{proof}

\smallskip

\begin{lemma}\label{l:finite-sides}
Assume that a recollement of the form
(\ref{eq:recollement-unbounded-der}) is given. The following hold
true.
\begin{itemize}
\item[(a)] \emph{(\cite[Proposition 3]{Han11})} Assume that $A$, $B$ and $C$ are projective over $k$. Then there exists a (unique) right bounded complex
$X$ of projective $C$-$A$-bimodules which as a
complex of $A$-modules is quasi-isomorphic to $T$, and up to
equivalence, we can assume
\[
j_!=?\lten_C X,~~ j^!=\rhom_A(X,?)=?\lten_A X^{\tr_A}\text{ and } j_*=\RHom_C(X^{\tr_A},?).\]
In
particular, $j_!j^!(A)= X^{\tr_A}\lten_C X$, and there is a
canonical triangle
\[
\xymatrix{
X^{\tr_A}\lten_C X= j_!j^!(A)\ar[r]& A\ar[r]& i_*i^*(A)\ar[r]&
j_!j^!(A)[1].
}
\]

Similarly, there exists a (unique) right bounded complex
$Y$ of projective $A$-$B$-bimodules which as a
complex of $B$-modules is quasi-isomorphic to $i^*(A)$, and up to
equivalence, we can assume
\[
i^*=?\lten_A Y,~~i_*=\rhom_B(Y,?)=?\lten_B Y^{\tr_B}\text{ and } i^!=\RHom_A(Y^{\tr_B},?).
\]

\item[(b)]  Suppose now that $k$ is a field and $A$ is a finite-dimensional $k$-algebra.
Then $B$ and $C$ are also finite-dimensional $k$-algebras. Moreover,
the functor $j_!$ restricts to a triangle functor $\cd^b(\mod
C)\rightarrow \cd^-(\mod A)$ and the objects $j_!j^!(A)$,
$i_*i^*(A)$ and $i_*(B)$ belong to $\cd^b(\mod A)$. Further, the object $X$ (respectively, $Y$) in (a) can be taken as a right bounded complex of finitely generated projective $C$-$A$-bimodules (respectively, $A$-$B$-bimodules).
\end{itemize}
\end{lemma}
\begin{proof}
(a) This was proved in \cite{Han11} for the case when $k$ is a field. But the proof there works in our more general setting.

(b) Recall that $T=j_!(C)\in\cd(\Mod A)$ is compact and exceptional
and has endomorphism algebra $C$. It follows that $C$ is a
finite-dimensional $k$-algebra.

Next we show that $j_!$ restricts to a triangle functor $\cd^b(\mod
C)\rightarrow \cd^-(\mod)$. Let $M\in \cd^b(\mod C)$. For
$n\in\mathbb{Z}$, we have
\begin{eqnarray*}
DH^n(j_!(M))&=&D\Hom_{\cd(\Mod A)}(A,j_!(M)[n])\\
 &\cong& \Hom_{\cd(\Mod A)}(j_!(M)[n],D(A))\\
 &\cong& \Hom_{\cd(\Mod C)}(M,j^!D(A)[-n]).
 \end{eqnarray*}
 Here the first isomorphism follows from the Auslander--Reiten
 formula and the second one follows by adjunction.
Since $j_!(C)=T\in K^b(\proj A)$,
Lemma~\ref{l:restriction-of-functors-to-compact-obj} and
Lemma~\ref{l:restriction-of-right-adjoint} imply that both $j^!(A)$ and
$j^!(D(A))$ belong to $\cd^b(\mod C)$. Consequently, the space
$\Hom_{\cd(\Mod A)}(M,j^!D(A)[-n])$ is finite-dimensional for each
$n$ and vanishes for sufficiently large $n$. Therefore $j_!(M)$ has
right bounded cohomologies, i.e. $j_!(M)\in\cd^-(\mod A)$.

In
particular $j_!j^!(A)\in\cd^-(\mod A)$. In view of (a), we have
$j_!j^!(A)\in\cd^b(\mod A)$. But then the canonical triangle
$X^{\tr}\lten_C X\cong j_!j^!(A)\rightarrow A\rightarrow i_*i^*(A)\rightarrow
j_!j^!(A)[1]$  and Lemma \ref{l:useful-info}(c)
yield that
$i_*i^*(A)$ and $i_*(B)$  belong to $\cd^b(\mod A)$ as well.
This also shows that
$B=\End_A(i_*(B))$ is finite-dimensional.

The last statement holds because the complexes $X$ and $Y$ have finite-dimensional total cohomology and the algebras $C^{op}\ten A$ and $A^{op}\ten B$ are finite-dimensional.
\end{proof}

\begin{remark} The converse of the first statement of
Lemma~\ref{l:finite-sides}(b) is not true: $B$ and $C$ being
finite-dimensional algebras does not imply that $A$ is finite-dimensional.
For an example, take $A$ to be the path algebra of the infinite
Kronecker quiver, i.e. the quiver which has two vertices $1$ and $2$
and which has infinitely many arrows from $1$ to $2$ and no arrows
from $2$ to $1$. Then the projective module $e_1A$ generates a
recollement
$$\xymatrix@!=6pc{\cd(\Mod A/Ae_1A) \ar[r]&\cd(\Mod A) \ar@<+2.5ex>[l]
\ar@<-2.5ex>[l] \ar[r] & \cd(\Mod e_1Ae_1)\ar@<+2.5ex>[l]
\ar@<-2.5ex>[l],}$$ where both $A/Ae_1A$ and $e_1Ae_1$ are
isomorphic to $k$, while $A$ is infinite-dimensional. This example
is taken from \cite[Example 9]{Koenig91}.
\end{remark}

\smallskip

\subsubsection{\sc Construction of
recollements}\label{reminder}\label{sss:construction-of-recollements}

For finite-dimensional algebras $A$, $B$ and $C$ over a field $k$,
every recollement of the form
\begin{align*}\label{eq:recollement-bounded-der}\xymatrix@!=6pc{\cd^b(\mod B) \ar[r]|{i_*=i_!} &\cd^b(\mod A)
\ar@<+2.5ex>[l]|{i^!} \ar@<-2.5ex>[l]|{i^*} \ar[r]|{j^!=j^*} &
\cd^b(\mod C)\ar@<+2.5ex>[l]|{j_*} \ar@<-2.5ex>[l]|{j_!}
}\tag{$R^b$}\end{align*} is given by a pair of compact exceptional objects
 $T=j_!(C),T'=i_*(B)\in \cd(\Mod A)$ with  $T'$ being right orthogonal to $T$,  see \cite[2.2 and 2.5]{AKL3}.
 A similar result holds true for recollements of the form (\ref{eq:recollement-unbounded-der}),
 as recalled in \ref{ss:recollement-of-der-cat}.
 In particular, in both cases the recollement is generated by the compact exceptional object  $T=j_!(C)$.

 Conversely,
every compact exceptional object $T\in \cd(\Mod A)$ with
$C=\End_A(T)$ generates a recollement  $$\xymatrix@!=6pc{\cd(\Mod B)
\ar[r]|{i_*=i_!} &\cd(\Mod A) \ar@<+2.5ex>[l]|{i^!}
\ar@<-2.5ex>[l]|{i^*} \ar[r]|{j^!=j^*} & \cd(\Mod
C)\ar@<+2.5ex>[l]|{j_*} \ar@<-2.5ex>[l]|{j_!} }
$$ where  $B=\rhom_A(T',T')$ for a certain object $T'\cong i_*
i^*(A)\in \cd(\Mod A)$    which occurs, up to isomorphism, in the
canonical triangle $j_!j^!(A)\rightarrow A\rightarrow i_*
i^*(A)\rightarrow j_!j^!(A)[1]$.

We  stress, however, that  $B$  is {\it just a dg algebra}, unless
the object $T'$ is exceptional, in which case  we obtain a
recollement of the form (\ref{eq:recollement-unbounded-der}) with
$B\cong\End_A(T')$.

 In the latter case, the recollement is  also
induced by a homological ring epimorphism $\lambda:A\to B$, that is, a ring epimorphism such that
$\Tor_i^A(B,B)=0$ for all $i>0$. Then, up to isomorphism,
  $$i_*=\lambda^*=?\lten_B B_A={\mathbf R} \Hom_B(_AB,?),\; i^*=?\lten_A B,\;
i^!={\mathbf R} \Hom_A(B_A,?),\text{ and  }j^*=?\lten_A X$$ with $X$ given by the triangle $X\to A\stackrel{\lambda}{\to} B\to X[1].$
For
details we refer to \cite[1.6 and 1.7]{AKL1}. A method for
determining $T'$ is given in \cite[Appendix]{AKL1}.

In the examples in Sections \ref{s:lifting-and-restricting-recollements}, \ref{s:derived-simplicity}, and \ref{JH}, we will often consider the following special case of the construction above. Let $T=eA$ where $e=e^2\in A$ is an idempotent such that $AeA$ is a \emph{stratifying ideal}, that is, $Ae\lten_{eAe} eA\cong AeA$, or equivalently, $\lambda:A\to A/AeA$ is a homological ring epimorphism. Then $T$ generates  a  recollement of the form (\ref{eq:recollement-unbounded-der}) where $C= eAe$ and $B= A/AeA$, and $T'\cong B_A$ occurs in the canonical triangle $AeA\to A\stackrel{\lambda}{\to} B\to AeA[1].$ Here the functors   $i_*,\,i^*,\,i^!$ are as above, and $$j_!=?\lten_{eAe} eA,\;j^*={\mathbf R}\text{Hom}_A(eA,?)=?\lten_A Ae,\,j_*={\mathbf R}\text{Hom}_{eAe}(Ae,?).$$

\begin{lemma} \label{l:criterion-stratifying-ideal}
Let $A$ be a $k$-algebra and $e\in A$ be an idempotent. Assume that $A/AeA$, as a (right) $A$-module, admits a projective resolution with components in $\add(eA)$ except in degree $0$. Then $AeA$ is a stratifying ideal of $A$ and the projection $A\ra A/AeA$ is a homological epimorphism.
\end{lemma}
\begin{proof}
The surjection $A\ra A/AeA$ is a ring epimorphism, and under the above assumptions,
$\Tor_i^A(A/AeA,A/AeA)=0$ for $i>0$.
The desired result follows.
\end{proof}

\smallskip
\subsubsection{\sc Derived simple algebras}\label{dersimple}
Derived simplicity (sometimes also called derived simpleness)
of an algebra was introduced by Wiedemann
\cite{W}. Let $\cd=K^b(\proj), \cd_{fl}, \cd^b(\mod), \cd^b(\Mod),
\cd^-(\Mod)$ or $\cd(\Mod)$. By definition, a $k$-algebra $A$ is
said to be {\em derived simple} with respect to $\cd$ (or
\emph{$\cd$-simple} for short) if there is no nontrivial $\cd$-recollement, namely, a recollement
of the form
$$\xymatrix@!=6pc{\cd_B \ar[r]|{i_*=i_!} &\cd_A \ar@<+2.5ex>[l]|{i^!}
\ar@<-2.5ex>[l]|{i^*} \ar[r]|{j^!=j^*} & \cd_C\ar@<+2.5ex>[l]|{j_*}
\ar@<-2.5ex>[l]|{j_!} }$$ where $B$ and $C$ are also $k$-algebras.
We reserve the short term {\em derived simple} for
$\cd(\mathrm{Mod})$-simple.

Fields, and more generally, local algebras are derived simple. The first
examples of derived simple algebras over a field with more than one
simple module have been constructed by Wiedemann \cite{W} and by
Happel \cite{Happel91}. In \cite{LY} derived simplicity has been
established for blocks of group algebras of finite groups, in any
characteristic, and also for indecomposable symmetric algebras of
finite representation type, provided that the base ring $k$ is a
field.

A $k$-algebra is said to be {\em indecomposable} if it is not
isomorphic to a direct product of two nonzero rings. A non-trivial
decomposition of a ring yields a non-trivial recollement. Hence a
decomposable ring never is derived simple in any sense.

\subsubsection{\sc Stratifications}\label{strat}
Having defined derived
simplicity, we can study stratifications.
Roughly speaking, a
\emph{stratification} is a way of breaking up a given derived
category into simple pieces using recollements. More rigorously, let $\cd=K^b(\proj), \cd_{fl}, \cd^b(\mod), \cd^b(\Mod),
\cd^-(\Mod)$ or $\cd(\Mod)$ and let $A$ be an algebra; a
stratification of $\cd_A$ (or a $\cd$-stratification of $A$) is a full rooted binary tree whose root is the given
derived category $\cd_A$, whose nodes are derived categories of type $\cd$ of algebras and whose
leaves are derived categories of type $\cd$ of $\cd$-simple algebras such that a node is a recollement of its two child
nodes unless it is a leaf.  The leaves are called the \emph{simple factors} of the
stratification. By abuse of language, we will also call the algebras whose derived categories are the leaves the simple factors of the stratification.

\medskip

\subsubsection{\sc Finiteness of global dimension}\label{sss:finiteness-of-global-dim}

A major reason for the interest in recollements and stratifications
of derived module categories is that the algebras $B$ and $C$
in the two outer terms frequently are less complicated than $A$.  One
can then study $A$ by investigating the  two outer algebras. This
reduction works well with respect to homological dimensions.  We will consider the global dimension.

Let $A$, $B$ and $C$ be finite-dimensional algebras over a field
forming a $\dbmod$-recollement of the form (\ref{eq:recollement-bounded-der}). Wiedemann
\cite[Lemma 2.1]{W} showed that $A$ has finite global dimension if
and only if so do $B$ and $C$. This was generalised in
\cite[Corollary 5]{Koenig91} to $\cd^-(\text{Mod})$-recollements of algebras over general commutative rings. For completeness we include a
detailed proof here.

\begin{prop}\emph{(\cite[Corollary 5]{Koenig91})} \label{fingldim}
Let $A$, $B$ and $C$ be
$k$-algebras and assume there is a recollement of the following type
\begin{align*}\xymatrix@!=6pc{\cd^-(\Mod B) \ar[r] &\cd^-(\Mod A) \ar@<+1.5ex>[l]
\ar@<-1.5ex>[l] \ar[r] & \cd^-(\Mod C) \ar@<+1.5ex>[l]
\ar@<-1.5ex>[l] }.\label{eq:recollement-bounded-above}\tag{$R^-$}\end{align*}Then $A$ is of finite global dimension if and
only if so are $B$ and $C$.
\end{prop}

\begin{proof}
We first observe the following
{\em Fact}: An algebra $A$ has finite
global dimension if and only if $\cd^b(\Mod A)=K^b(\Proj A)$. If $A$
has finite global dimension, the equality holds by definition. If
$A$ has infinite global dimension,
there is a module of infinite
projective dimension  (if there are infinitely many $n\in\N$ and
modules $M_n$ with $\pd (M_n)=n$, take the
direct sum of those $M_n$), which is in the bounded
derived category, but not in $K^b(\Proj A)$. This proves the fact.

Now suppose a recollement on $\cd^-(\Mod)$ level is given. If $A$
has finite global dimension, by \cite[Proposition 4]{Koenig91} the
given recollement can be restricted to recollements on $\cd^b(\Mod)$
level and on $K^b(\Proj)$ level. We compare the two recollements:
the middle terms coincide, the left hand terms satisfy $K^b(\Proj B)
\subset \cd^b(\Mod B)$ and the right hand terms also satisfy
$K^b(\Proj C) \subset \cd^b(\Mod C)$. But the inclusion $K^b(\Proj
B) \subset \cd^b(\Mod B)$ in the left hand side of the recollement
implies the converse inclusion $K^b(\Proj C) \supset \cd^b(\Mod C)$
on the right hand side of the recollement. Therefore equality holds
and $C$ and also $B$ have finite global dimension. Conversely if $B$
and $C$ have finite global dimension, the same argument of
restricting recollements works.
\end{proof}

Based on this result, we prove the analogue for $\cd(\Mod)$-recollements.

\begin{prop}\label{p:finiteness-of-global-dimension}
In a recollement of the form (\ref{eq:recollement-unbounded-der}) the algebra $A$ is of finite
global dimension if and only if so are $B$ and $C$.
\end{prop}

\begin{proof} We will show below that if $\gldim(A)<\infty$ or
$\gldim(C)<\infty$ then $i_*(B)\in K^b(\Proj A)$. As
$j_!(C)\in K^b(\proj A)$, ~\cite[Theorem
1]{Koenig91} implies the existence of a recollement
$$\xymatrix@!=6pc{\cd^-(\Mod B) \ar[r]&\cd^-(\Mod A) \ar@<+1.5ex>[l]
\ar@<-1.5ex>[l] \ar[r] & \cd^-(\Mod C)\ar@<+1.5ex>[l]
\ar@<-1.5ex>[l]}.$$ Applying the previous proposition then
concludes the proof of the assertion.

What is left to show is that $i_*(B)\in K^b(\Proj A)$.
If $\gldim(A)<\infty$, then by Lemma~\ref{l:useful-info} (e) the
object $i_*(B)\in\cd^b(\Mod A)\cong K^b(\Proj A)$. Suppose
$\gldim(C)<\infty$. Again using Lemma~\ref{l:useful-info} (e)
yields $j^*(A)\in \cd^b(\Mod C)=K^b(\Proj C)$. On the
other hand, since $j_!$ commutes with direct sums and $j_!(C)\in
K^b(\proj A)$, it follows that $j_!(\Proj C)\subseteq K^b(\Proj A)$.
Hence $j_!(K^b(\Proj C))\subseteq K^b(\Proj A)$. Therefore,
$j_!j^*(A)\in K^b(\Proj A)$, and hence $i_*i^*(A)\in K^b(\Proj A)$
due to the triangle
\[\xymatrix{j_!j^*(A)\ar[r]& A\ar[r] & i_*i^*(A)\ar[r] & j_!j^*(A)[1].}\]
By Lemma~\ref{l:useful-info} (c), $i_*(B)\in K^b(\Proj A)$.
\end{proof}

\section{Ladders}\label{s:ladder}

When considering all possible recollements for a fixed triangulated
or derived category, it is convenient to group them, for
example, into ladders. A \emph{ladder} $\cl$ is a finite or infinite
diagram of triangulated categories and triangle functors
\vspace{-10pt}\[{\setlength{\unitlength}{0.7pt}
\begin{picture}(200,170)
\put(0,70){$\xymatrix@!=3pc{\mathcal{C}'
\ar@<+2.5ex>[r]|{i_{n-1}}\ar@<-2.5ex>[r]|{i_{n+1}} &\mathcal{C}
\ar[l]|{j_n} \ar@<-5.0ex>[l]|{j_{n-2}} \ar@<+5.0ex>[l]|{j_{n+2}}
\ar@<+2.5ex>[r]|{j_{n-1}} \ar@<-2.5ex>[r]|{j_{n+1}} &
\mathcal{C}''\ar[l]|{i_n} \ar@<-5.0ex>[l]|{i_{n-2}}
\ar@<+5.0ex>[l]|{i_{n+2}}}$} 
\put(52.5,10){$\vdots$} 
\put(137.5,10){$\vdots$} 
\put(52.5,130){$\vdots$} 
\put(137.5,130){$\vdots$} 
\end{picture}}
\]
such that any three consecutive rows form a recollement. The rows are labelled by a subset of $\mathbb{Z}$ and multiple occurence of the same recollement is allowed. This
definition is taken from~\cite[Section
1.5]{BeilinsonGinzburgSchechtman88} with a minor modification. The \emph{height} of a ladder is the number of recollements
contained in it (counted with multiplicities). It is an element of $\mathbb{N}\cup \{0,\infty\}$. A recollement
is considered to be a ladder of height 1.

The ladder $\cl$ induces a \emph{TTF tuple} $(\mathrm{Im} i_n)_n$, and is
uniquely determined, up to equivalence, by any entry of the tuple.
A ladder can be \emph{unbounded}, \emph{bounded above},
\emph{bounded below} or \emph{bounded} in the obvious
sense, and the corresponding TTF tuple is \emph{unbounded},
\emph{left bounded}, \emph{right bounded} or \emph{bounded}.
The ladder $\cl$ is a \emph{ladder of
derived categories} if $\cc'$, $\cc$ and $\cc''$ are derived
categories of algebras. A ladder is \emph{complete} if it is not a
proper subladder of any other ladder. Then the induced
TTF tuple also is said to be complete.

\begin{lemma}\label{l:complete-ladder} Let $\cc$ be a compactly generated triangulated
category. Let $(\ldots,\cc_{-2},\cc_{-1},\cc_0)$ be a TTF tuple of
$\cc$, and assume that $\cc_{-1}$ is compactly generated by one
object. Then this TTF tuple is complete if and only if some (or any)
compact generator of $\cc_{-1}$ is not compact in $\cc$.
\end{lemma}
\begin{proof} The pair $(\cc_{-1},\cc_0)$ can be
completed to a TTF triple $(\cc_{-1},\cc_0,\cc_1)$ if and only if a
(or any) compact generator of $\cc_{-1}$ is compact in $\cc$,
see~\cite[Lemma 4.4.8 and Remark 5.2.6]{NZ}.
\end{proof}

\begin{proposition}\label{p:ladder-extension}
Let $A$, $B$ and $C$ be $k$-algebras forming a recollement
\begin{align*}\xymatrix@!=6pc{\cd(\Mod B) \ar[r]|{i_*=i_!} &\cd(\Mod A) \ar@<+2.5ex>[l]|{i^!}
\ar@<-2.5ex>[l]|{i^*} \ar[r]|{j^!=j^*} & \cd(\Mod
C)\ar@<+2.5ex>[l]|{j_*} \ar@<-2.5ex>[l]|{j_!} }.
\label{eq:recollement-unbounded-ladder}\tag{$R$}\end{align*}
\begin{itemize}
\item[(a)] The recollement
 can be
extended one step downwards (in the obvious sense) if and only
if $j_*$ (equivalently $i^!$) has a right adjoint. This occurs precisely
when $j^!$ (equivalently $i_*$) restricts to $K^b(\proj)$.
\item[(b)] The recollement
can be extended one step upwards if and only if $j_!$
(equivalently $i^*$) has a left adjoint. If $A$ is a finite-dimensional algebra
over a field, this occurs precisely when $j_!$  (equivalently $i^*$) restricts
to $\cd^b(\mod)$.
\end{itemize}
\end{proposition}
\begin{proof} The first statements of both (a) and (b) follow from
Lemma~\ref{l:complete-ses-to-t-str}. Further, it follows from Lemma~\ref{l:complete-ladder}  that the recollement
 can be
extended downwards  if and only if $i_*(B)$ is compact, which means by Lemma~\ref{l:restriction-of-functors-to-compact-obj} that $i_*$ restricts to $K^b(\proj)$. This proves the second statement
of (a) concerning $i_*$; the equivalence with the statement in parentheses
is part of Lemma \ref{l:restriction-to-compact-for-row} below, which
will be proven without using any result of the present section.
The second statement
of (b) is just the equivalence of (ii) and (iii) in  Lemma~\ref{l:existence-of-left-adjoint}: Indeed, both $i^*$ and $j_!$ are derived tensor products, by Lemma \ref{l:finite-sides} (a). Moreover, the left adjoints of $i^*$ and $j_!$ always restrict to $K^b(\proj)$ by Lemma~\ref{l:useful-info} (e), since they form the top row of  a recollement.
\end{proof}

Proposition~\ref{p:ladder-extension} can be extended by induction as follows. Let $X$ and $Y$ be as in Lemma~\ref{l:finite-sides}(a) and define $X_n$, $Y_n$ recursively by setting $X_0=X$, $Y_0=Y$ and
\[
X_n = \begin{cases}
(X_{n-1})^{\tr_A} & \text{if } n \text{ is odd and } n>0\\
(X_{n-1})^{\tr_C} & \text{if } n \text{ is even and } n>0\\
(X_{n+1})^{\tr_{C^{op}}} & \text{if } n \text{ is odd and } n<0\\
(X_{n+1})^{\tr_{A^{op}}} & \text{if } n \text{ is even and } n<0
\end{cases},~~
Y_n =\begin{cases}
(Y_{n-1})^{\tr_B} & \text{if } n \text{ is odd and } n>0\\
(Y_{n-1})^{\tr_A} & \text{if } n \text{ is even and } n>0\\
(Y_{n+1})^{\tr_{A^{op}}} & \text{if } n \text{ is odd and } n<0\\
(Y_{n+1})^{\tr_{B^{op}}} & \text{if } n \text{ is even and } n<0.
\end{cases}.
\]
All of these are complexes of bimodules.

\begin{proposition}\label{p:unbounded-ladder}
Suppose that $k$ is a field and $A$, $B$, $C$ are finite-dimensional $k$-algebras forming a recollement of the form (\ref{eq:recollement-unbounded-der}). Let $m\in\mathbb{N}$.
\begin{itemize}
\item[(a)] The following statements are equivalent:
\begin{itemize}
\item[(i)] (\ref{eq:recollement-unbounded-der}) can be extended $m$ steps downwards,
\item[(ii)] $(X_{n})_A$ is compact in $\cd(\Mod A)$ for $n$ even with $1\leq n\leq m$ and $(X_{n})_C$ is compact in $\cd(\Mod C)$ for $n$ odd with $1\leq n\leq m$,
\item[(iii)] $(Y_{n})_B$ is compact in $\cd(\Mod B)$ for $n$ even with $1\leq n\leq m$ and $(Y_{n})_A$ is compact in $\cd(\Mod A)$ for $n$ odd with $1\leq n\leq m$.
\end{itemize}
If this is the case, we obtain a ladder of height $m+1$. We label the rows of this ladder by $0,\ldots,m+2$ from the top. For $n$ even (respectively, odd) with $0\leq n\leq m+1$, the two functors of the $n$-th row are $?\lten_A Y_{n}$ and $?\lten_C X_{n}$ (respectively, $?\lten_B Y_{n}$ and $?\lten_A X_{n}$). The two functors of the $(m+2)$-nd row are $\RHom_B(Y_{m+1},?)$ and $\RHom_A(X_{m+1},?)$ if $m$ is odd and $\RHom_A(Y_{m+1},?)$ and $\RHom_C(X_{m+1},?)$ if $m$ is even.

\item[(b)] The following statements are equivalent:
\begin{itemize}
\item[(i)] (\ref{eq:recollement-unbounded-der}) can be extended $m$ steps upwards,
\item[(ii)] ${}_C(X_{n})$ is compact in $\cd(C\mbox{-}\Mod)$ for $n$ even with $1-m\leq n\leq 0$ and ${}_A(X_{n})$ is compact in $\cd(A\mbox{-}\Mod)$ for $n$ odd with $1-m\leq n\leq 0$,
\item[(iii)] ${}_A(Y_{n})$ is compact in $\cd(A\mbox{-}\Mod)$ for $n$ even with $1-m\leq n\leq 0$ and ${}_B(Y_{n})$ is compact in $\cd(B\mbox{-}\Mod)$ for $n$ odd with $1-m\leq n\leq 0$.
\end{itemize}
If this is the case, we obtain a ladder of height $m+1$. We label the rows of this ladder by $-m,\ldots,0,1,2$ from the top. For $n$ even (respectively, odd) with $-m\leq n\leq 1$, the two functors of the $n$-th row  are $?\lten_A Y_{n}$ and $?\lten_C X_{n}$ (respectively, $?\lten_B Y_{n}$ and $?\lten_A X_{n}$). The two functors of the second row are $\RHom_A(Y_{1},?)$ and $\RHom_C(X_{1},?)$.
\end{itemize}
\end{proposition}
\begin{proof}
This follows by induction on $m$ from Proposition~\ref{p:ladder-extension}, Lemma~\ref{l:finite-sides} and Lemma~\ref{l:existence-of-left-adjoint}. We leave the details to the reader.
\end{proof}

Next we give some examples of ladders and TTF tuples.

\begin{ex} {\rm Let $k$ be a field, $B$ and $C$ be finite-imensional $k$-algebras and $M$ a finitely generated $C$-$B$-bimodule. Consider the matrix algebra $A=\left(\begin{array}{cc} B & 0\\ {}_CM_B & C\end{array}\right)$. Put $e_1=\left(\begin{array}{cc} 1 & 0\\ 0 & 0\end{array}\right)$ and $e_2=\left(\begin{array}{cc} 0 & 0\\ 0 & 1\end{array}\right)$. Using Lemma~\ref{l:criterion-stratifying-ideal}, one checks that both $Ae_1A$ and $Ae_2A$ are stratifying ideals of $A$. Moreover, they produce a ladder of height $2$
\[
\xymatrix@C=5pc{\cd(\Mod B)
\ar@<+1ex>[r]|{?\ten_B^{\mathbf L} e_1A} \ar@<-3ex>[r] &\cd(\Mod A) \ar@<1ex>[l]|{?\ten_A Ae_1}
\ar@<-3ex>[l]
 \ar@<+1ex>[r] \ar@<-3ex>[r] & \cd(\Mod
C).\ar@<1ex>[l] \ar@<-3ex>[l]|{?\ten_C^{\mathbf L} e_2A}}
\]
By Proposition~\ref{p:ladder-extension}(b) and Lemma~\ref{l:existence-of-left-adjoint}, this ladder can be extended one step upwards if and only if ${}_C(e_2A)\in K^b(C\mbox{-}\proj)$ if and only if ${}_CM$ has finite projective dimension over $C$. By Proposition~\ref{p:ladder-extension}(a) and Lemma~\ref{l:restriction-of-functors-to-compact-obj}, it can be extended one step downwards if and only if $(Ae_1)_B\in K^b(\proj B)$ if and only if $M_B$ has finite projective dimension over $B$.
}
\end{ex}

\begin{ex} {\rm Let $k$ be a field and let $A$ be the path algebra of the quiver $\xymatrix{1\ar[r] & 2}$. Up to isomorphism and shift there are three indecomposable exceptional objects in $K^b(\proj A)$, given by the following representations
\[P_1=(\xymatrix{k&0\ar[l]}),~~ P_2=(\xymatrix{k&k\ar[l]_{id}}),~~ S_2= (\xymatrix{0&k\ar[l]}).\]
They generate an unbounded TTF tuple which is periodic of period $3$:
\[(\ldots,\Tria(P_2),\Tria(P_1),\Tria(S_2),\Tria(P_2),\Tria(P_1),\ldots).\]
Note that the TTF triple $(\Tria(P_2),\Tria(P_1),\Tria(S_2))$ corresponds to the recollement generated by $P_2=e_2A$, while the TTF triple $(\Tria(P_1),\Tria(S_2),\Tria(P_2))$ corresponds to the recollement generated by $P_1=e_1A$, where $e_1$ and $e_2$ are trivial paths at $1$ and $2$, respectively.}
\end{ex}

\begin{ex} \label{ex:s(2,2)}
{\rm Let $k$ be a field and let $A$ be the algebra given by the quiver with relations
$$\xymatrix{1\ar@<.7ex>[r]^{\alpha}&2\ar@<.7ex>[l]^{\beta}},~~\alpha\beta.$$
This is a quasi-hereditary algebra (with respect to the ordering $1<2$); its global
dimension is two.

An object of $K^b(\proj A)$ is indecomposable if and
only it is isomorphic to one of the following complexes
\begin{itemize}
\item[$\cdot$] $\xymatrix{P_1\ar[r]^{\beta\alpha} & P_1\ar@{.}[r] &
P_1\ar[r]^{\beta\alpha}&P_1}$,
\item[$\cdot$] $\xymatrix{P_2\ar[r]^{\beta}& P_1\ar[r]^{\beta\alpha} &P_1\ar@{.}[r] &
P_1\ar[r]^{\beta\alpha}&P_1}$,
\item[$\cdot$] $\xymatrix{P_1\ar[r]^{\beta\alpha} &P_1\ar@{.}[r] &
P_1\ar[r]^{\beta\alpha}&P_1\ar[r]^\alpha & P_2}$,
\item[$\cdot$] $\xymatrix{P_2\ar[r]^{\beta}& P_1\ar[r]^{\beta\alpha} &P_1\ar@{.}[r] &
P_1\ar[r]^{\beta\alpha}&P_1 \ar[r]^{\alpha} & P_2}$,
\end{itemize}
see for example~\cite{BobinskiGeissSkowronski04}.
The objects in the second and third families are exceptional with endomorphism algebra $k$, the projective $P_1$ is exceptional with endomorphism algebra $k[x]/x^2$, while
the other objects in the first and fourth families are not exceptional.

The projective
module $P_2$ generates a $\cd(\Mod)$-recollement (this is the one
from the quasi-hereditary structure of $A$), which sits in the
infinite ladder corresponding to the TTF tuple
$(\cc_n)_{n\in\mathbb{Z}}$, where
\[\cc_n=\begin{cases}\Tria(\nu^{\frac{n}{2}}P_2) & \text{ if $n$ is even},\\
\Tria(\nu^{\frac{n-1}{2}}S_1) & \text{ if $n$ is odd}.\end{cases}\]
Here $\nu$ is the derived Nakayama functor for $A$.
One directly checks that
up to shift all exceptional indecomposable objects of $K^b(\proj A)$ other than $P_1$
already occur in the tuple
\[(\ldots,\nu^{-1}P_2,\nu^{-1}S_1,P_2,S_1,\nu P_2,\nu S_1,\ldots).\]

To construct this ladder explicitly, one may start with the recollement given by the quasi-hereditary structure and first compute the outer terms: The algebra on the left hand side is $B:=A/Ae_2A \simeq k$. The algebra on the right hand side is $C:=e_2Ae_2 \simeq k$. Next, the images of $B$ and $C$, respectively, can be determined:
 $i_{\ast}(B) \simeq B \lten_B B_A \simeq S_1$ is simple, $j_!(C) \simeq
 e_2Ae_2 \lten_{e_2Ae_2} e_2A \simeq e_2 A$ is projective and
$j_{\ast}(C) \simeq \RHom_{e_2Ae_2}(Ae_2,e_2Ae_2) \simeq \Hom_k(Ae_2,k)$
is the injective $A$-module $I_2$. Observing that $I_2 \simeq \nu P_2$ shows that
the exceptional objects $P_2, S_1, \nu P_2$ correspond to three subsequent
recollements in the ladder, which therefore is determined by iteratively applying $\nu$ and its inverse to these objects.

Observe that the endomorphism algebra of $P_1$ has infinite global dimension, while $A$ has finite global dimension. It follows from Proposition~\ref{p:finiteness-of-global-dimension} that $P_1$ cannot generate a recollement of derived module categories. Therefore up to equivalence the above ladder is the unique
non-trivial ladder of derived categories for $\cd(\Mod A)$. Note
that the above tuple of objects, read from the right to the left, is
a helix in the sense of the Rudakov school \cite{Rudakov90}.}
\end{ex}

In Example~\ref{ex:s(2,2)} the ladder has a repeating pattern given by the Nakayama functor. This is a general phenomenon, compare \cite{Jorgensen10,Zhangpu11}.

\begin{proposition} Keep the notation and assumptions as in Proposition~\ref{p:unbounded-ladder}. Assume further that $A$ has finite global dimension and let $\nu:\cd(\Mod A)\rightarrow\cd(\Mod A)$ be the derived Nakayama functor for $A$. Set $T=j_!(C)$ and $T'=i_*(B)$. Then (\ref{eq:recollement-unbounded-der}) sits in an unbounded ladder which corresponds to the TTF tuple $(\cc_n)_{n\in\mathbb{Z}}$, where
\[
\cc_n=\begin{cases}\Tria(\nu^{\frac{n}{2}}T) & \text{ if $n$ is even},\\
\Tria(\nu^{\frac{n-1}{2}}T') & \text{ if $n$ is odd}.\end{cases}
\]
\end{proposition}
\begin{proof}
By Proposition~\ref{fingldim}, both $B$ and $C$ have finite global dimension. Then it follows by induction on $n$ that for all $n\in\mathbb{Z}$ the complex $X_n$ is compact in $\cd(\Mod A)$, $\cd(\Mod C)$, $\cd(C\mbox{-}\Mod)$ or $\cd(A\mbox{-}\Mod)$, depending on the sign and parity of $n$. Thus by Proposition~\ref{p:unbounded-ladder}, the given recollement (\ref{eq:recollement-unbounded-der}) sits in an unbounded ladder which corresponds to the TTF tuple $(\cc_n)_{n\in\mathbb{Z}}$, where
\[
\cc_n=\begin{cases}\Tria((X_n)_A) & \text{ if $n$ is even},\\
\Tria((Y_n)_A) & \text{ if $n$ is odd}.\end{cases}
\]
It remains to show $\Tria((X_n)_A)=\Tria(\nu^{\frac{n}{2}}T)$ if $n$ is even and $\Tria((Y_n)_A)=\Tria(\nu^{\frac{n-1}{2}}T')$ if $n$ is odd. Up to shift of the TTF tuple, it suffices to show $\Tria((X_2)_A)=\Tria(\nu T)$. By definition, $X_2=(X^{\tr_A})^{\tr_C}=\RHom_C(\RHom_A(X,A),C)$. Since $C$ has finite global dimension, the projective generator $C$ and injective cogenerator $DC$ generate each other in $\cd(\Mod C)$ in finitely many steps. Therefore, $(X_2)_A$ and $\nu(T)\cong\nu(X)=D\RHom_A(X,A)\cong\RHom_C(\RHom_A(X,A),DC)$ generate each other in $\cd(\Mod A)$ in finitely many steps. In particular, $\Tria((X_2)_A)=\Tria(\nu(T))$, as desired.
\end{proof}

\section{Lifting and restricting recollements}\label{s:lifting-and-restricting-recollements}

Recollements of derived module categories can be defined on all
levels --- bounded or unbounded derived categories and finitely
generated or general module categories. While lifting of
recollements from bounded to left or right bounded level and from
left or right bounded to unbounded level is not problematic (see
\cite{Koenig91} and \cite[Section 4]{AKL2}), going in the opposite
direction is not always possible. In general, recollements on
unbounded level need not restrict to recollements on bounded or left
or right bounded level. In the first and second subsections,
characterisations are given, when lifting and restricting is
possible. In the third subsection, an example is provided to
illustrate the conditions in these characterisations.

\smallskip

\subsection{Lifting recollements to $\cd(\Mod)$}

Let $A$ be a $k$-algebra. Let $\cd_A$ be one of the following
derived categories associated to $A$: $\cd^-(\Mod A)$, $\cd^b(\Mod
A)$, $\cd^b(\mod A)$ (for $k$ being a field and $A$ being finite-dimensional over
$k$) and $K^b(\proj A)$.

\begin{proposition}\label{p:lifting-recollement} Let $A$, $B$ and $C$ be $k$-algebras.
\begin{itemize}
\item[(a)] \emph{(\cite[Lemma 4.1, 4.2, 4.3]{AKL2} and~\cite[Corollary
2.7]{AKL3})} Any recollement
\begin{eqnarray}\label{eq:recollement-on-subcat}\xymatrix@!=6pc{\cd_B
\ar[r]|{i'_*=i'_!} &\cd_A \ar@<+2.5ex>[l]|{i'^!}
\ar@<-2.5ex>[l]|{i'^*} \ar[r]|{j'^!=j'^*} &
\cd_C\ar@<+2.5ex>[l]|{j'_*} \ar@<-2.5ex>[l]|{j'_!} }\end{eqnarray}
can be lifted to a recollement
\begin{eqnarray}\label{eq:lifted-recollement-on-unbounded}\xymatrix@!=6pc{\cd(\Mod B) \ar[r]|{i_*=i_!} &\cd(\Mod A)
\ar@<+2.5ex>[l]|{i^!} \ar@<-2.5ex>[l]|{i^*} \ar[r]|{j^!=j^*} &
\cd(\Mod C)\ar@<+2.5ex>[l]|{j_*} \ar@<-2.5ex>[l]|{j_!}
}\end{eqnarray} such that $j_!(C)\cong j'_!(C)$,
$i_*(B)\cong i'_*(B)$ and $j_*(C)\cong j'_*(C)$.
\item[(b)] The lifted recollement (\ref{eq:lifted-recollement-on-unbounded}) in (a) restricts,
up to equivalence,  to the recollement
(\ref{eq:recollement-on-subcat}).
\end{itemize}
\end{proposition}
\begin{proof}
(b) We will show that the TTF triple
$(\cx,\cy,\cz)$ associated with the lifted recollement
(\ref{eq:lifted-recollement-on-unbounded}) restricts to the TTF
triple $(\cx_1,\cy_1,\cz_1)$ associated with the given recollement
(\ref{eq:recollement-on-subcat}). The desired result then follows
from Proposition~\ref{p:recollement-corresponds-to-TTF}. Recall that
$\cx=\Tria(j_!(C))$, $\cy=\Tria(i_*(B))=(\tria(j_!(C)))^\perp$ and
$\cz=(\tria(i_*(B)))^\perp$. It is clear that $\cx_1\subseteq \cx$.
Moreover, $j_!(C)\in\cx_1$ being left orthogonal to $\cy_1$ implies
$\cy_1\subseteq \cy$. Similarly, $\cz_1\subseteq\cz$. Consequently,
$\cx_1\subseteq\cx\cap\cd_A$, $\cy_1\subseteq \cy\cap\cd_A$ and
$\cz_1\subseteq\cz\cap\cd_A$. In particular, $\cx_1\perp
(\cy\cap\cd_A)$, which implies that $\cy\cap\cd_A\subseteq \cy_1$
because $(\cx_1,\cy_1)$ is a $t$-structure of $\cd_A$. Therefore the
equality $\cy_1=\cy\cap\cd_A$ holds. Similarly
$\cx_1=\cx\cap\cd_A$ and $\cz_1=\cz\cap\cd_A$.
\end{proof}

\begin{rem} It is not clear if the restriction of the
induced recollement (\ref{eq:lifted-recollement-on-unbounded}) coincides
with (instead of being equivalent to)
 the given
recollement (\ref{eq:recollement-on-subcat}). This would positively answer
Rickard's question asking whether any derived equivalence is a
derived tensor functor given by a bimodule complex.
\end{rem}

\subsection{Restricting recollements from $\cd(\Mod)$}
 Suppose we are
given three $k$-algebras $A$, $B$ and $C$ together with a
recollement of the form
\begin{align*}\label{eq:recollement-on-unbounded-in-section-restricting}\xymatrix@!=6pc{\cd(\Mod B) \ar[r]|{i_*=i_!}
&\cd(\Mod A) \ar@<+2.5ex>[l]|{i^!} \ar@<-2.5ex>[l]|{i^*}
\ar[r]|{j^!=j^*} & \cd(\Mod C)\ar@<+2.5ex>[l]|{j_*}
\ar@<-2.5ex>[l]|{j_!} }\tag{$R$}\end{align*} We are interested in
conditions under which this recollement can be restricted to a
recollement on $K^b(\proj)$, $\cd^b(\mod)$ (when $k$ is a field and
the algebras are finite-dimensional $k$-algebras), $\cd^b(\Mod)$ and
$\cd^-(\Mod)$.

\subsubsection{\sc Restricting recollements to $K^b(\proj)$}

We start with an auxiliary result.
\begin{lemma}\label{l:non-compact-to-non-compact}
Let $F:\cc\rightarrow\cc'$ be a fully faithful triangle functor
commuting with direct sums. If $X\in\cc$ is not compact, then $F(X)$
is not compact.
\end{lemma}
\begin{proof} Let $\{Y_i\mid i\in I\}$ be a set of objects of $\cc$. There is a commutative diagram
\[\xymatrix{\Hom_\cc(X,\bigoplus_{i\in I}Y_i)\ar[r]^(0.43)\simeq & \Hom_{\cc'}(F(X),F(\bigoplus_{i\in I}Y_i))\\
&\Hom_{\cc'}(F(X),\bigoplus_{i\in I}F(Y_i))\ar[u]^\simeq\\
\bigoplus_{i\in I}\Hom_\cc(X,Y_i)\ar[uu]\ar[r]^(0.43){\simeq} & \bigoplus_{i\in I}\Hom_{\cc'}(F(X),F(Y_i)).\ar[u]}\]
If $F(X)$ is compact in $\cc'$, then the right lower map is bijective, implying that the left map is bijective, i.e. $X$ is compact in $\cc$.
\end{proof}

Next we show that $i_*$ restricts to $K^b(\proj)$ if and only if so
does $j^*$. Moreover, if this happens, then $i^!$ restricts to
$K^b(\proj)$ if and only if so does $j_*$.

\begin{lemma}\label{l:restriction-to-compact-for-row}
 \begin{itemize}
  \item[(a)] $i_*(B)\in K^b(\proj A)$ if and only if $j^*(A)\in K^b(\proj C)$, which happens if and only if $j_!j^*(A)\in K^b(\proj A)$.
  \item[(b)] Assume that $i_*(B)\in K^b(\proj A)$. Then $i^!(A)\in K^b(\proj B)$ if and only if $j_*(C)\in K^b(\proj A)$.
 \end{itemize}
\end{lemma}
\begin{proof} (a)
Consider the canonical triangle
\[\xymatrix{j_!j^*(A)\ar[r] & A\ar[r] & i_*i^*(A)\ar[r] &j_!j^*(A)[1]}.\]
If $i_*(B)\in K^b(\proj A)$, then $i_*i^*(A)\in K^b(\proj A)$ by
Lemma~\ref{l:useful-info} (c). Therefore $j_!j^*(A)\in K^b(\proj
A)$, and hence $j^*(A)\in K^b(\proj A)$ by
Lemma~\ref{l:non-compact-to-non-compact}. Conversely, if $j^*(A)\in K^b(\proj C)$, then $j_!j^*(A)\in K^b(\proj A)$ by Lemma~\ref{l:useful-info}(e). So the above triangle implies that $i_*i^*(A)\in K^b(\proj A)$. Now by Lemma~\ref{l:useful-info}(c), $i_*(B)\in K^b(\proj A)$.

(b) Assume that $i_*(B)\in K^b(\proj A)$ and $i^!(A)\in K^b(\proj
B)$ hold. Consider the following canonical triangle associated to
$j_!(C)$
\[\xymatrix@R=0.5pc{i_*i^!j_!(C)\ar[r] & j_!(C)\ar[r] & j_*j^*j_!(C)\ar[r]\ar@{=}[d] & i_*i^!j_!(C)[1].\\
&&j_*(C)&}\] It follows from the assumption that $i_*i^!j_!(C)\in
K^b(\proj A)$ since $j_!(C)\in K^b(\proj A)$. Thus $j_*(C)\in
K^b(\proj A)$.

Assume that $i_*(B)\in K^b(\proj A)$ and $j_*(C)\in K^b(\proj A)$
hold. Consider the following canonical triangle associated to $A$:
\[\xymatrix{i_*i^!(A)\ar[r] & A\ar[r] & j_*j^*(A)\ar[r] & i_*i^!(A)[1].}\]
Applying $i^*$ to this triangle, we obtain
\[\xymatrix@R=0.5pc{i^*i_*i^!(A)\ar[r]\ar@{=}[d] & i^*(A)\ar[r] & i^*j_*j^*(A)\ar[r] & i_*i^!(A)[1].\\
i^!(A)}\] We know from (a) and the assumption that $j^*(A)\in
K^b(\proj C)$. Since $j_*(C)\in K^b(\proj A)$ and $i^*$ restricts to
$K^b(\proj)$, it follows that $i^*j_*j^*(A)\in K^b(\proj B)$, and
hence $i^!(A)\in K^b(\proj B)$.
\end{proof}

\begin{thm}\label{t:recollement-restricting-to-kbproj} The following are equivalent:
\begin{itemize}
\item[(i)] the recollement (\ref{eq:recollement-on-unbounded-in-section-restricting}) restricts to a
recollement
$$\xymatrix@!=6pc{K^b(\proj B) \ar[r]|{i_*=i_!} &K^b(\proj A) \ar@<+2.5ex>[l]|{i^!}
\ar@<-2.5ex>[l]|{i^*} \ar[r]|{j^!=j^*} & K^b(\proj C),
\ar@<+2.5ex>[l]|{j_*} \ar@<-2.5ex>[l]|{j_!} }$$
\item[(ii)] $i_*(B)\in K^b(\proj A)$ and $i^!(A)\in
K^b(\proj B)$,
\item[(iii)] $j^*(A)\in
K^b(\proj C)$ and $j_*(C)\in K^b(\proj A)$.
\end{itemize}
\end{thm}

\begin{proof} By  Lemma~\ref{l:useful-info}(e) both functors $j_!$ and $i^*$ can be restricted
to $K^b(\proj)$. Therefore a
  recollement of the form (\ref{eq:recollement-on-unbounded-in-section-restricting}) restricts to
$K^b(\proj)$ if and only if $i_*(B)\in K^b(\proj A)$, $i^!(A)\in K^b(\proj
B)$, $j^*(A)\in K^b(\proj C)$ and $j_*(C)\in K^b(\proj A)$. The
desired result then follows from
Lemma~\ref{l:restriction-to-compact-for-row}.
\end{proof}

Example~\ref{e:three-recollements} (c) will show that requiring only the condition
$i^!(A)\in K^b(\proj A)$ in (ii) (or $B$ has finite global
dimension, \confer  ~Corollary~\ref{cor:reco-restriction-C} below) is
not enough for the recollement
(\ref{eq:recollement-on-unbounded-in-section-restricting}) to
restrict to $K^b(\proj)$. The next example shows that requiring only the
condition $j_*(C)\in K^b(\proj A)$ in (iii) also is not sufficient.

\begin{ex}\label{ex:large-tilting}{\rm Let $k$ be a field, $T=k^{\oplus \mathbb{N}}$ and
$A=\End_k(T)$. Then elements in $A$ are identified with
$\mathbb{N}\times\mathbb{N}$-matrices such that there are only
finitely many non-zero entries in each column. Let $e$ be the
elementary matrix with entry $1$ in position $(1,1)$ and with entry
$0$ in any other position. Then $AeA$ is a stratifying ideal and
hence yields a recollement
$$\xymatrix@C=6pc{\cd(\Mod A/AeA) \ar[r]|{i_*=i_!} &\cd(\Mod A) \ar@<+2.5ex>[l]|{i^!}
\ar@<-2.5ex>[l]|{i^*} \ar[r]|{j^!=j^*} & \cd(\Mod
k).\ar@<+2.5ex>[l]|{j_*} \ar@<-2.5ex>[l]|{j_!} }$$
which is even a stratification. Indeed, $AeA$ consists of the endomorphisms of $T$ with finite-dimensional image, so $ A/AeA$ is a simple von Neumann regular ring  by \cite[4.27]{Lam} and \cite[the second Example on p. 16]{Clark}  and therefore derived simple by \cite[4.11]{AKL2}. One checks that
$j_*(k)=eA$ belongs to $K^b(\proj A)$, while $j^*(A)=T$ does not
belong to $K^b(\proj k)$. So this recollement does not restrict to $K^b(\proj)$, although $j_*$ restricts to $K^b(\proj)$.

More generally, this phenomenon occurs  in every recollement induced by a good tilting module. Recall that a module $T_R$ over a ring $R$ is said to be a  {\it good tilting module} if it has finite projective dimension, $\Ext_R^i(T,T^{(\alpha)})=0$ for any cardinal $\alpha$,
and there is an exact sequence of right $R$-modules
$0\ra R\ra T_0\ra\ldots\ra T_n\ra 0$
with $T_0,T_1,\ldots T_n\in \add T$. It was shown by
 Chen and Xi in \cite{CX1,CX4} that  $T_R$ then
   induces a recollement $$\xymatrix@C=6pc{\cd(\Mod B) \ar[r]|{i_*=i_!} &\cd(\Mod A) \ar@<+2.5ex>[l]|{i^!}
\ar@<-2.5ex>[l]|{i^*} \ar[r]|{j^!=j^*} & \cd(\Mod
R).\ar@<+2.5ex>[l]|{j_*} \ar@<-2.5ex>[l]|{j_!} }$$
where $A=\End_R(T)$, $j_*=\RHom_R(T,?)$, and  $j^*=?\lten_A T$.

Since $R$ is quasi-isomorphic to the complex $T_0\ra\ldots\ra T_n$, we see that  $j_*(R)\in K^b(\proj A)$, so  $j_*$ restricts to $K^b(\proj)$
by Lemma \ref{l:restriction-of-functors-to-compact-obj}. On the other hand,  $j^*$ restricts to $K^b(\proj)$ if and only if $j^*(A)\cong T_R\in K^b(\proj R)$, that is, if and only if $T_R$ is a classical tilting module.
}
\end{ex}

\subsubsection{\sc Restricting recollements to $\cd^b(\mod)$} Let $k$ be a
field. Suppose we are given a recollement of the form
(\ref{eq:recollement-on-unbounded-in-section-restricting}) where $A$, and thus also $B$ and $C$, are finite-dimensional $k$-algebras. Recall that in this case $\cd_{fl}$
coincides with $\cd^b(\mod)$.

Let $X$ and $X^{\tr}=\rhom_A(X,A)$ be chosen as in Lemma~\ref{l:finite-sides}(a).

\begin{thm}\label{t:restricting-recollements-to-dbmod}
The following are equivalent for a finite-dimensional algebra $A$:
\begin{itemize}
\item[(i)] the recollement (\ref{eq:recollement-on-unbounded-in-section-restricting})
restricts to a recollement
$$\xymatrix@!=6pc{\cd^b(\mod B) \ar[r]|{i_*=i_!} &\cd^b(\mod A) \ar@<+2.5ex>[l]|{i^!}
\ar@<-2.5ex>[l]|{i^*} \ar[r]|{j^!=j^*} & \cd^b(\mod
C),\ar@<+2.5ex>[l]|{j_*} \ar@<-2.5ex>[l]|{j_!} }$$
\item[(ii)] ${}_CX\in K^b(C\mbox{-}\proj)$ and $(X^{\tr}\lten_C X)_A\in K^b(\proj A)$;
\item[(iii)] ${}_CX\in K^b(C\mbox{-}\proj)$ and $X^{\tr}_C\in K^b(\proj
C)$.
\item[(iv)]   $j_!$ restricts to $\cd^b(\mod)$ and $i_*(B)\in K^b(\proj A)$.
\end{itemize}\label{l:restriction-reco}
\end{thm}

\begin{proof} (i) $\Rightarrow$ (ii): Suppose the recollement can be restricted to the level of $\cd^b(\mod)$.
Applying Lemma \ref{l:existence-of-left-adjoint} to $j_!=?\lten_C
X$ yields that ${}_CX\in K^b(C-\proj)$. Furthermore by
\cite[Corollary 2.4]{AKL3}, the functors $j^!$ and $j_!$ send
compact objects to compact objects. In particular, $j_!j^!(A)\cong
X^{\tr}\lten_C X$, as a complex of right $A$-modules, is compact.

The equivalence of (ii), (iii), and (iv)  holds true by Lemma~\ref{l:existence-of-left-adjoint} and by Lemma~\ref{l:restriction-to-compact-for-row} (a) since $j^*(A)=X^{\tr}_C$
and $j_!j^*(A)=(X^{\tr}\lten_C X)_A$.

(iv) $\Rightarrow$ (i): $i_*$ and $j^*$ always restrict
to $\cd^b(\mod)$ by Lemma~\ref{l:useful-info}(e), and $j_!$ restricts to
$\cd^b(\mod)$ by assumption.
Moreover, since $i_*(B)\in
K^b(\proj A)$, it follows from Lemma~\ref{l:restriction-of-functors-to-compact-obj} and
Lemma~\ref{l:restriction-of-right-adjoint} that $i^!$ restricts to
$\cd^b(\mod)$. Similarly, one shows that $j_*$  restricts
to $\cd^b(\mod)$ since $j^*(A)=X^{\tr}_C\in K^b(\proj C)$.
It remains to check $i^*$.
 For $M\in\cd^b(\mod A)$, there is a triangle
\[\xymatrix{j_!j^*(M)\ar[r] & M\ar[r] & i_*i^*(M)\ar[r] & j_!j^*(M)[1].}\]
Since both $j_!$ and $j^*$ restricts to $\cd^b(\mod)$, the object $j_!j^*(M)$, and hence $i_*i^*(M)$, are in $\cd^b(\mod A)$. Because $i_*$ is
a full embedding, there are isomorphisms
\[H^n(i^*(M))=\Hom_B(B,i^*(M)[n])=\Hom_A(i_*(B),i_*i^*(M)[n]).\]
Since $i_*(B)\in K^b(\proj A)$ and $i_*i^*(M)\in\cd^b(\mod A)$, it
follows that $i^*(M)$ has bounded finite-dimensional cohomologies,
that is, $i^*(M)\in\cd^b(\mod B)$.
\end{proof}

A special case in which condition (iii) is satisfied is when the
algebra $C$ has finite global dimension. Indeed, since $X_A$ is
compact in $\cd(\Mod A)$, as a complex $X$ has finite-dimensional
total cohomological space. Hence $_C X \in \cd^b(C$-mod$)$. When $C$
has finite global dimension, $_CX$ is compact in
$\cd(C$-Mod$)$. A similar argument applies to $X^{\tr}$.

\begin{cor}\label{cor:reco-restriction-C}
Given a  finite-dimensional algebra $A$, a recollement of the form
(\ref{eq:recollement-on-unbounded-in-section-restricting}) can be
restricted to $\cd^b(\mod)$ provided the algebra $C$ has finite
global dimension.
\end{cor}
However, this corollary fails when replacing $C$ by $B$.
Counterexamples will be given below in Example~\ref{e:three-recollements}.

\subsubsection{\sc Restricting recollements to $\cd^b(\Mod)$}

The following result is included in the proof of~\cite[Proposition
4]{Koenig91} and in~\cite[Lemma 4.1, Example 4.6]{AKL2}.

\begin{proposition}\label{p:recollement-restricting-to-dbMod}
The following are equivalent  for a $k$-algebra $A$:
\begin{itemize}
\item[(i)] the recollement (\ref{eq:recollement-on-unbounded-in-section-restricting}) restricts to a recollement
$$\xymatrix@!=6pc{\cd^b(\Mod B) \ar[r]|{i_*=i_!} &\cd^b(\Mod A) \ar@<+2.5ex>[l]|{i^!}
\ar@<-2.5ex>[l]|{i^*} \ar[r]|{j^!=j^*} & \cd^b(\Mod
C),\ar@<+2.5ex>[l]|{j_*} \ar@<-2.5ex>[l]|{j_!} }$$
\item[(ii)] $i_*(B)\in K^b(\Proj A)$ and $j_!$ restricts to
$\cd^b(\Mod)$.
\end{itemize}
\end{proposition}

\begin{corollary}\label{c:restricting-is-equivalent-for-dbMod-and-dbmod-for-fd-alg} Let $k$ be a field and let $A$ be a finite-dimensional
$k$-algebra. Then the recollement
(\ref{eq:recollement-on-unbounded-in-section-restricting}) restricts
to $\cd^b(\Mod)$ if and only if it restricts to $\cd^b(\mod)$.
\end{corollary}
\begin{proof}
By Theorem~\ref{t:restricting-recollements-to-dbmod} and Proposition~\ref{p:recollement-restricting-to-dbMod}, we have to show that
the following  are equivalent:
{\it
\begin{itemize}
\item[(a)] $i_*(B)\in K^b(\Proj A)$ and $j_!$ restricts to
$\cd^b(\Mod)$.
\item[(b)] $i_*(B)\in K^b(\proj A)$ and $j_!$ restricts to
$\cd^b(\mod)$.
\end{itemize}
}
Condition (a) implies (b) by Lemma~\ref{l:finite-sides} (b), which shows that
$i_*(B)\in \cd^b(\mod A)$ and $j_!$ restricts to $\cd^b(\mod
C)\rightarrow\cd^-(\mod A)$.
Conversely,
let $X$ be as in Lemma~\ref{l:finite-sides}(a)
such that
$j_!=?\lten_C X$. Under condition (b), the equivalence
(ii)$\Leftrightarrow$(iii) of
Lemma~\ref{l:existence-of-left-adjoint} and the implication
(i)$\Rightarrow$(ii) of Lemma~\ref{l:restriction-of-right-adjoint} imply
that $j_!$ restricts to $\cd^b(\Mod)$. This completes the proof.
\end{proof}

There is a $\cd^b(\Mod)$-counterpart of
Corollary~\ref{cor:reco-restriction-C}.

\begin{corollary}\label{c:restricting-to-dbMod-when-C-has-finite-global-dimension}
If $C$ has finite global dimension, then the recollement
(\ref{eq:recollement-on-unbounded-in-section-restricting}) restricts
to $\cd^b(\Mod)$.
\end{corollary}
\begin{proof}
Assume that $C$ has finite global dimension. In the proof of Proposition~\ref{p:finiteness-of-global-dimension} it has been shown that $i_*(B)\in K^b(\Proj A)$ and $j_!$ restricts to $K^b(\Proj)$. Since $\cd^b(\Mod C)\cong K^b(\Proj C)$ and $K^b(\Proj A)\subseteq \cd^b(\Mod A)$, it follows that $j_!$ restricts to $\cd^b(\Mod)$. The desired result is obtained by applying Proposition~\ref{p:recollement-restricting-to-dbMod}.
\end{proof}

\subsubsection{\sc Restricting recollements to $\cd^-(\Mod)$}
\begin{proposition}\label{p:recollement-restricting-to-d-Mod} The following
are equivalent:
\begin{itemize}
\item[(i)] the recollement (\ref{eq:recollement-on-unbounded-in-section-restricting}) restricts to a recollement
$$\xymatrix@!=6pc{\cd^-(\Mod B) \ar[r]|{i_*=i_!} &\cd^-(\Mod A) \ar@<+2.5ex>[l]|{i^!}
\ar@<-2.5ex>[l]|{i^*} \ar[r]|{j^!=j^*} & \cd^-(\Mod
C),\ar@<+2.5ex>[l]|{j_*} \ar@<-2.5ex>[l]|{j_!} }$$
\item[(ii)] $i_*(B)\in K^b(\Proj A)$.
\end{itemize}
If $k$ is a field and $A$ is finite-dimensional over $k$, then both
conditions are equivalent to
\begin{itemize}
\item[(iii)] $i_*(B)\in K^b(\proj A)$.
\end{itemize}
\end{proposition}
\begin{proof} The equivalence between (i) and (ii) is \cite[Lemma
4.4]{AKL2}. The rest is as in the proof of
Corollary~\ref{c:restricting-is-equivalent-for-dbMod-and-dbmod-for-fd-alg}.
\end{proof}

\subsubsection{\sc Special case: $A$ has finite global dimension}

\begin{proposition}\label{p:restricting-for-algebra-of-finite-global-dimension}
If $A$ has finite global dimension, then the recollement
(\ref{eq:recollement-on-unbounded-in-section-restricting}) restricts
to $\cd^-(\Mod)$ and $\cd^b(\Mod)$. If in addition $k$ is a field
and $A$ is a finite-dimensional $k$-algebra, the recollement
(\ref{eq:recollement-on-unbounded-in-section-restricting}) restricts
to $\cd^b(\mod)$ as well.
\end{proposition}
\begin{proof} The first statement follows from the proof of
Proposition~\ref{p:finiteness-of-global-dimension}
and Corollary~\ref{c:restricting-to-dbMod-when-C-has-finite-global-dimension}.
The second statement follows from
Corollary~\ref{c:restricting-is-equivalent-for-dbMod-and-dbmod-for-fd-alg}.
\end{proof}

\subsection{An example}

The following example \ref{e:three-recollements} illustrates, in particular,
the lack of symmetry in Corollary \ref{cor:reco-restriction-C}.
There are three $\cd(\Mod)$-recollements. One of them cannot be
restricted to any of $\cd^-(\Mod)$, $\cd^b(\Mod)$, $\cd^b(\mod)$ or
$K^b(\proj)$. The other two restrict to $\cd^-(\Mod)$, with one
further restricting to $\cd^b(\Mod)$ and $\cd^b(\mod)$ and the other
one further restricting to $K^b(\proj)$.

\begin{ex}\label{e:three-recollements}
{\rm Let $k$ be a field and $A$ be the $k$-algebra given by quiver and relations
\[\xymatrix{1 & 2\ar[l]_{\alpha}\ar@(ur,dr)^{\beta}},~~\beta^2,~\alpha\beta.\]
 Denote by
$P_i=e_iA$ and $S_i$, respectively, the indecomposable projective
and simple module at the vertex $i=1,2$. So $P_1$ has composition
series $\begin{smallmatrix} 1\\2\end{smallmatrix}$ and $P_2$ has
composition series $\begin{smallmatrix} 2\\2\end{smallmatrix}$.

As explained in Section \ref{reminder}, any  recollement is generated by a compact exceptional object. So we start out by showing that  $P_1$, $P_2$ and
$\Cone(P_2\stackrel{\alpha}{\rightarrow}P_1)$ are the only indecomposable exceptional
  objects in $K^b(\proj A)$ up to shift.

Let $X$ be an indecomposable exceptional object in $K^b(\proj A)$. Its terms
are direct sums of copies of $P_1$  and of $P_2$. If the first and the last non-zero
term have a common summand, the identity map on this summand gives a
morphism, which is not homotopic to zero, from $X$ to a shifted copy of $X$.
There are no non-zero maps from $P_1$ to $P_2$, and non-zero
endomorphisms of $P_1$ are isomorphisms. Thus, if $X$ is chosen minimal and has
length at least two, it cannot start with copies of $P_1$.
Therefore, we may assume that $X$ is minimal of the form
{
\[\xymatrix{P_2^{\oplus n_a}\ar[r]&P_1^{\oplus m_{a+1}}\oplus P_2^{\oplus n_{a+1}}\ar[r]&\ldots\ar[r]&P_1^{\oplus m_{b}}\oplus P_2^{\oplus n_{b}}\ar[r] & P_1^{\oplus m_{b+1}}},\]
}
where $a$ and $b$ are integers indicating the degrees, and $n_a$ and
{$n_{b}$}
are positive. Let $f:P_2^{\oplus n_{b-1}}\rightarrow P_2^{\oplus n_a}$ be a nonzero morphism which, in matrix form, has entries from {$e_2Je_2$.}
Then
\[\xymatrix@C=1pc@R=1pc{P_2^{\oplus n_a}\ar[r]&P_1^{\oplus m_{a+1}}\oplus P_2^{\oplus n_{a+1}}\ar[r]&\ldots\ar[r]&P_1^{\oplus m_{b-1}}\oplus P_2^{\oplus n_{b-1}}\ar[d]^{(0,f)}\ar[r] & P_1^{\oplus m_b}\ar[d]^0\\
&&&P_2^{n_a}\ar[r]&P_1^{\oplus m_{a+1}}\oplus P_2^{\oplus n_{a+1}}\ar[r]&\ldots}\]
is a chain map which is not homotopic to zero. Therefore $\Hom(X,X[a-b])\neq 0$. Since $X$ is exceptional, this forces $a=b$. It follows that up to shift $X$ is isomorphic to one of the stated objects.

\smallskip

(a) Consider the recollement generated by $P_1=e_1A$, which has endomorphism algebra
$\End_A(P_1)=e_1Ae_1\cong k$. As a right $A$-module $A/Ae_1A$ is isomorphic to $P_2$. So by Lemma~\ref{l:criterion-stratifying-ideal}, the
canonical map $A\rightarrow A/Ae_1A\cong k[x]/x^2$ is a homological
epimorphism, and the recollement has the form
\[\xymatrix@C=6pc{\cd({\Mod}A/Ae_1A) \ar[r]|{i_*=i_!} &\cd(\Mod A) \ar@<+2.5ex>[l]|{i^!}
\ar@<-2.5ex>[l]|{i^*} \ar[r]|{j^!=j^*} &
\cd({\Mod}e_1Ae_1).\ar@<+2.5ex>[l]|{j_*} \ar@<-2.5ex>[l]|{j_!}}
\]
Then \begin{itemize}
\item[--] $i_*(A/Ae_1A)=P_2\in K^b(\proj A)$,
\item[--] the algebra on the right hand side $e_1Ae_1\cong k$ has global dimension $0$,
\item[--] $j_*(e_1Ae_1)=S_1\not\in K^b(\proj A)$. \end{itemize}
It follows that the recollement restricts to $\cd^-(\Mod)$
(Proposition~\ref{p:recollement-restricting-to-d-Mod}),
$\cd^b(\Mod)$
(Corollary~\ref{c:restricting-to-dbMod-when-C-has-finite-global-dimension})
and $\cd^b(\mod)$ (Corollary~\ref{cor:reco-restriction-C}), but not
to $K^b(\proj)$ (Theorem~\ref{t:recollement-restricting-to-kbproj}).

\smallskip

(b) Consider the recollement generated by
$M=\Cone(P_2\stackrel{\alpha}{\rightarrow}P_1)$. The endomorphism
ring $\End(M)$ of $M$ is $k[x]/x^2$, where $x$ is represented by the
endomorphism
 \[\xymatrix@R=1.5pc{P_2\ar[d]^{\beta}\ar[r]^{\alpha}&P_1\ar[d]^{0}\\
 P_2\ar[r]^{\alpha}&P_1.}\]
The object $M$ admits a real action (not up to homotopy, actually there are
no homotopies) by $\End(M)$ on the left.

In order to determine the recollement, we consider the following triangle:
\[
\xymatrix{
M[-1]\ar[r]& A=P_1\oplus P_2\ar[r]& P_1\oplus P_1\ar[r]& M.
}
\]
Since $M[-1]\in\Tria M$ and $P_1\oplus P_1\in(\Tria M)^\perp$, and canonical
triangles are unique, this must be the canonical triangle of $A$. It also
can be found by the method explained in \cite[Appendix]{AKL1}.

Clearly $P_1\oplus P_1=i^*i_*(A)$ is compact and exceptional with endomorphism
algebra $M_2(k)$, the algebra of $2\times 2$-matrices with entries
in $k$. It follows from \cite[Proposition 1.7]{AKL1} that the recollement generated by $M$ is induced by a
homological epimorphism $\lambda: A\to M_2(k)$ and has the form
\[\xymatrix@!=9pc{\cd(\Mod M_2(k)) \ar[r]|{i_*=i_!} &\cd(\Mod A) \ar@<+2.5ex>[l]|{i^!}
\ar@<-2.5ex>[l]|{i^*} \ar[r]|{j^!=j^*} & \cd(\Mod
\End(M)),\ar@<+2.5ex>[l]|{j_*} \ar@<-2.5ex>[l]|{j_!}}
\]
where the left column is induced from $\lambda$, cf.~Section \ref{sss:construction-of-recollements}. Moreover, since $M$ is a complex of $\End(M)$-$A$-bimodules,
$j_!\cong ? \lten_{\End(M)}M$. Then
\begin{itemize}
\item[--] it
follows from the canonical triangle that $M[-1]=j_!j^!(A)\in
K^b(\proj A)$;
\item[--]  $i_*(M_2(k))=P_1\oplus P_1\in K^b(\proj A)$, by Lemma \ref{l:useful-info}(c), using that  $M_2(K)$ is semisimple;
\item[--] ${}_{\End(M)}M\not\in
K^b(\End(M)\mbox{-}\proj)$, because $M \cong k\oplus k[1]$ over $\End(M)$,
where $k$ is the unique simple $\End(M)$-module;
\item[--] $i^!(A)={\mathbf R}\text{Hom}_A(M_2(k),A)\cong{\mathbf R}\text{Hom}_A(P_1\oplus P_1,A) \in K^b(\proj M_2(k))$.
\end{itemize}
It follows that the recollement restricts to $\cd^-(\Mod)$
(Proposition~\ref{p:recollement-restricting-to-d-Mod}) and
$K^b(\proj)$ (Theorem~\ref{t:recollement-restricting-to-kbproj}),
but not to $\cd^b(\Mod)$, nor to $\cd^b(\mod)$
(Theorem~\ref{t:restricting-recollements-to-dbmod} and
Corollary~\ref{c:restricting-is-equivalent-for-dbMod-and-dbmod-for-fd-alg}).
This is an example where $B$ has finite global dimension but the
recollement cannot be restricted to $\cd^b(\mod)$.

\smallskip

(c) Consider the recollement generated by $P_2=e_2A$, whose endomorphism algebra is
$\End_A(P_2)=e_2Ae_2\cong k[x]/x^2$.
As a right $A$-module $A/Ae_2A$ is isomorphic to $S_1$ and it has a projective resolution
\[\xymatrix{\ldots\ar[r]&P_2\ar[r]&\ldots\ar[r] & P_2\ar[r] & P_1\ar[r] & A/Ae_2A\ar[r] & 0}.\]
Therefore by Lemma~\ref{l:criterion-stratifying-ideal} the canonical projection $\mu:A\rightarrow A/Ae_2A \cong k$ is a
homological epimorphism. The recollement has the form
\[\xymatrix@!=9pc{\cd(\Mod k) \ar[r]|{i_*=i_!} &\cd(\Mod A) \ar@<+2.5ex>[l]|{i^!}
\ar@<-2.5ex>[l]|{i^*} \ar[r]|{j^!=j^*} & \cd(\Mod
k[x]/x^2)\ar@<+2.5ex>[l]|{j_*} \ar@<-2.5ex>[l]|{j_!}}
\] where the left column is induced by the projection $\mu$ and
$j_!=?\lten_k P_2$. Then
\begin{itemize}
\item[--] ${}_{e_2Ae_2}P_2\cong e_2Ae_2\in K^b(e_2Ae_2\mbox{-}\proj)$,
\item[--] $i_*(k)=S_1\not\in K^b(\proj
A)$, \item[--] $i^!(A)=k\in K^b(\proj k)$.\end{itemize} It follows
that this recollement cannot be restricted to $\cd^-(\Mod)$
(Proposition~\ref{p:recollement-restricting-to-d-Mod}),
$\cd^b(\Mod)$
(Proposition~\ref{p:recollement-restricting-to-dbMod}),
$\cd^b(\mod)$ (Theorem~\ref{t:restricting-recollements-to-dbmod}), nor to
$K^b(\proj)$ (Theorem~\ref{t:recollement-restricting-to-kbproj}).
This shows that $i^!(A)$ being compact or even the algebra on the left
side having finite global dimension is not enough for the recollement
to restrict to $K^b(\proj)$.

\smallskip

To summarise, up to equivalence there are only three non-trivial
recollements of $\cd(\Mod A)$ by derived module categories.
In fact, the three
recollements can be put into one (complete) ladder of height $3$
\vspace{5pt}
\[\xymatrix@C=5pc{\cd(\Mod k)
\ar@<+1.5ex>[r] \ar@<-1.5ex>[r] &\cd(\Mod A) \ar[l] \ar@<-3.0ex>[l]
\ar@<+3.0ex>[l] \ar@<+1.5ex>[r] \ar@<-1.5ex>[r] & \cd(\Mod
k[x]/x^2),\ar[l] \ar@<-3.0ex>[l] \ar@<+3.0ex>[l]}
\]
\vspace{5pt}
which corresponds to the TTF quintuple
\[(\Tria(M),\Tria(P_1),\Tria(P_2),\Tria(S_1),\Tria(S_1)^\perp).\]
}
\end{ex}

\bigskip

\subsection{A correction}\label{corrigendum}
We take the opportunity to correct a mistake  from \cite{AKL3}. In \cite[Section 1.3]{AKL3}, it is stated that a homological ring epimorphism $\varphi:A\to B$ between two finite-dimensional algebras always induces a recollement of $\cd^b(\mod A)$ by $\cd^b(\mod B)$ and a triangulated category $\mathcal X$. This is then used in \cite[Theorem 3.3]{AKL3} for showing that a finitely generated tilting module over a finite-dimensional algebra $A$ induces a  recollement of $\cd^b(\mod A)$ by the bounded derived categories $\cd^b(\mod B)$ and $\cd^b(\mod C)$ of two algebras $B,C$ constructed from $T$.

Unfortunately, the assumption that $B$ has finite projective dimension both as a right $A$-module and as a left $A$-module is missing in both statements.
In fact, without this assumption the statements fail: a counterexample is provided by  the recollement (b) in Example \ref{e:three-recollements}  (take the ring epimorphism $\lambda:A\to M_2(k)$ and the tilting module $T= A$ with the $T$-resolution $0\to A\to A \oplus e_2 A \to e_2A \to 0$).

We remark that the statements were later employed in the context when $A$ has finite global dimension. Then it holds naturally that $\pd(B_A)<\infty$ and $\pd(_AB)<\infty$, thus the remaining results of \cite{AKL3} are not affected by this mistake.

Let us now prove the statements under the additional assumption. Take a homological ring epimorphism $\varphi:A\to B$ between two finite-dimensional algebras. It induces a recollement at $\cd(\Mod)$-level
\[
\xy (-61,0)*{\cd(\Mod B)}; {\ar (-34,3)*{}; (-50,3)*{}_{i^{\ast}}}; {\ar
(-50,0)*{}; (-34,0)*{}|{ i_{\ast}=i_!}}; {\ar(-34,-3)*{};
(-50,-3)*{}^{i^!}}; (-22,0)*{\cd(\Mod A)}; {\ar (3,3)*{}; (-11,3)*{}};
{\ar (-11,0)*{}; (3,0)*{}}; {\ar(3,-3)*{}; (-11,-3)*{}};
(8,0)*{\mcx};
\endxy
\]
where $i_\ast$ restricts to $\cd^b(\mod)$.  Combining
Lemma~\ref{l:restriction-of-right-adjoint} and \ref{l:restriction-of-functors-to-compact-obj}, we see that $i^!$ restricts to $\cd^b(\mod)$ if and only if $i_\ast$ restricts to $K^b(\proj)$, which in turn means that $B_A$ is a compact object. Moreover, we infer from Lemma~\ref{l:existence-of-left-adjoint} that $i^\ast=?\lten_AB$ restricts to $\cd^b(\mod)$ if and only if $_AB$ is compact.   Hence under the assumption that both $B_A$ and $_A B$ have finite projective dimension, the left part of the above recollement restricts to $\cd^b(\mod)$.


We briefly recall the situation of \cite[Theorem 3.3]{AKL3}. Let $A$ be a finite-dimensional algebra over a field $k$. Fix a tilting module $T$ in $\mod A$ together with a $T$-resolution of $A$ (not necessarily minimal)
$$0 \ra A \ra T_0 \ra T_1 \ra 0 \quad \text{with}\quad T_0,T_1 \in \add(T).$$
Let $\varphi: A \ra B$ be the universal localisation of $A$ at $T_1$, that is, $B = \End_A(T_0/ \text{tr}_{T_1}(T_0))$ and $\varphi$ is a ring epimorphism. Moreover let $C = \End_A(T_1)$.

\begin{thm}(Correction of \cite[Theorem 3.3]{AKL3}) Let $A$, $B$ and $C$ be as above. Under the following conditions
\begin{itemize}
\item[i)] $\varphi: A \ra B$ is a homological epimorphism,

\item[ii)] $\pd(_C T_1) < \infty$, or equivalently $\pd(_A B) <\infty$,

\item[iii)] $\pd(B_A) < \infty$,

\end{itemize}
there is a recollement at $\cd^b(mod)$-level
$$\xymatrix@!=6pc{\cd^b(\mod B) \ar[r]|{i_*=i_!} &\cd^b(\mod A) \ar@<+2.5ex>[l]|{i^!}
\ar@<-2.5ex>[l]|{i^*} \ar[r]|{j^!=j^*} & \cd^b(\mod
C)\ar@<+2.5ex>[l]|{j_*} \ar@<-2.5ex>[l]|{j_!} }$$
where $i^*=-\lten_A B$, $i_*=
\varphi_*$, $i^!=\rhom_A(B,-)$, $j_!=-\lten_C T_1$, and
$j^!=\rhom_A(T_1,-)$.
\end{thm}

\begin{proof} We proceed as in \cite{AKL3} to obtain a recollement at $\cd(\Mod)$-level with the desired functors, using that $\varphi$ is a homological epimorphism. 
By Theorem \ref{t:restricting-recollements-to-dbmod} it restricts to $\cd^b(\mod)$-level if and only if $j_!$ restricts to $\cd^b(\mod)$ and $i_*(B) \in K^b(\proj A)$. The later condition means by Lemma \ref{l:restriction-of-functors-to-compact-obj} that $B_A \in K^b(\proj A)$, that is, $\pd (B_A) < \infty$.
Let us now turn to the condition on $j_!$. By the second part of Proposition \ref{p:ladder-extension} (b) the functor $j_!$ restricts to $\cd^b(\mod)$ if and only if so does $i^*$. By Lemma \ref{l:existence-of-left-adjoint} this is further equivalent to $\pd (_C T_1) < \infty$, and also to $\pd(_A B) < \infty$. Altogether, the conditions i)-iii) therefore ensure the existence of a recollement as stated.
\end{proof}

\bigskip

\section{Derived simplicity}\label{s:derived-simplicity}

In this section, the term ladder refers to a ladder of unbounded derived module
categories. A ladder is \emph{trivial} if one of its three terms is
trivial. We characterise
derived simplicity with respect to different choices of derived categories
in terms of heights of ladders (Theorems~\ref{t:dminus-simple-and-ladder},~\ref{t:kbproj-simplicity-and-ladder} and~\ref{t:dbmod-simplicity-and-ladder}). Using this characterisation, it
is shown by examples that the concept of derived simplicity depends
on the choice of derived categories.

Finally, in Section \ref{ss:comm-ring} we exhibit a large family of derived simple algebras: all indecomposable commutative rings.

\subsection{Restricting recollements along ladders}

Recall that in Example~\ref{e:three-recollements} there is a ladder
of height $3$. The upper recollement of the ladder restricts to
$K^b(\proj)$, the middle recollement restricts to $\cd^b(\Mod)$ and
$\cd^b(\mod)$, and both the upper and the middle recollements
restrict to $\cd^-(\Mod)$. This is a general phenomenon.

Let $k$ be a commutative ring and suppose $A$, $B$ and $C$ are $k$-algebras
forming a recollement
\begin{align*}\label{eq:recollement-unbounded-in-section-derived-simplicity}
\xymatrix@!=6pc{\cd(\Mod B) \ar[r]|{i_*=i_!} &\cd(\Mod A)
\ar@<+2.5ex>[l]|{i^!} \ar@<-2.5ex>[l]|{i^*} \ar[r]|{j^!=j^*} &
\cd(\Mod C).\ar@<+2.5ex>[l]|{j_*} \ar@<-2.5ex>[l]|{j_!}
}\tag{$R$}\end{align*} Recall that the two functors in the upper row always
restrict to $K^b(\proj)$. Following~\cite{Han11}, we say that this
recollement is \emph{perfect} if $i_*(B)\in K^b(\proj A)$ holds.
By Proposition~\ref{p:recollement-restricting-to-d-Mod},
a perfect recollement restricts to $\cd^-(\Mod)$.

A ladder of height $2$
\[\xymatrix@C=4pc{\cd(\Mod B)
\ar@<+.75ex>[r] \ar@<-2.25ex>[r] &\cd(\Mod A) \ar@<.75ex>[l]
\ar@<-2.25ex>[l]
 \ar@<+.75ex>[r] \ar@<-2.25ex>[r] & \cd(\Mod
C),\ar@<.75ex>[l] \ar@<-2.25ex>[l]}
\]
contains two recollements, which we call the upper and the lower
recollement in the obvious sense.

\begin{lemma}\label{l:ladder-of-height-2}
Let $\cl$ be a ladder of height $2$. Then the upper recollement is a
perfect recollement. In particular, it restricts to $\cd^-(\Mod)$.
\end{lemma}
\begin{proof} The middle row of the upper
recollement is the upper row of the lower recollement, and hence
both functors in this row restrict to $K^b(\proj)$.
\end{proof}

A ladder of height $3$
\[\xymatrix@!=7pc{\cd(\Mod B)
\ar@<+1.5ex>[r] \ar@<-1.5ex>[r] &\cd(\Mod A) \ar[l] \ar@<-3.0ex>[l]
\ar@<+3.0ex>[l] \ar@<+1.5ex>[r] \ar@<-1.5ex>[r] & \cd(\Mod C),\ar[l]
\ar@<-3.0ex>[l] \ar@<+3.0ex>[l]}
\]
contains three recollements, which we call the upper, the middle and
the lower recollement in the obvious sense.

\begin{proposition}\label{p:ladder-of-height-3}
Let $\cl$ be a ladder of height $3$. Then the upper recollement of
$\cl$ restricts to $K^b(\proj)$ and the middle recollement of $\cl$
restricts to $\cd^b(\Mod)$ and $\cd_{fl}$.
\end{proposition}
\begin{proof}
The three rows of the upper recollement are respectively the upper
row of the upper, the middle and the lower recollement of $\cl$, and
hence all six functors in these three rows restrict to
$K^b(\proj)$. So, the upper recollement restricts to
$K^b(\proj)$. Being the right adjoints of these six functors, the
functors in the middle recollement restrict to $\cd^b(\Mod)$ and
$\cd_{fl}$, by Lemma~\ref{l:restriction-of-right-adjoint}. This
shows that the middle recollement of $\cl$ restricts to
$\cd^b(\Mod)$ and $\cd_{fl}$.
\end{proof}

\subsection{$\cd(\Mod)$-simplicity}
Being $\cd(\Mod)$-simple is the strongest property among all
versions of derived simplicity.

\begin{proposition}\label{p:unbounded-simplicity-is-strongest}
Let $A$ be $\cd(\Mod)$-simple. Then $A$ is derived simple with
respect to any of $\cd^-(\Mod)$, $\cd^b(\Mod)$, $\cd_{fl}$ or
$K^b(\proj)$.
\end{proposition}
\begin{proof}
Let $\cd=\cd^-(\Mod),\cd^b(\Mod),\cd_{fl}$ or $K^b(\proj)$. Suppose
that $A$ is not $\cd$-simple. Then there is a non-trivial
recollement of $A$ on the level $\cd$. By
Proposition~\ref{p:lifting-recollement}, there is a non-trivial
recollement of $A$ on the level of $\cd(\Mod)$, contradicting the
assumption.
\end{proof}

In general the concepts of derived simplicity with respect to $\cd(\Mod)$,
$\cd^-(\Mod)$, $\cd^b(\Mod)$, $\cd_{fl}$ and $K^b(\proj)$ are
different. However, for algebras of finite global dimension, some of
them coincide, as a consequence of
Proposition~\ref{p:restricting-for-algebra-of-finite-global-dimension}:

\begin{cor} Let $A$ be a $k$-algebra of finite global dimension. Then
$A$ is $\cd(\Mod)$-simple if and only if it is $\cd^-(\Mod)$-simple
if and only if it is $\cd^b(\Mod)$-simple. If in addition $k$ is a
field and $A$ is finite-dimensional over $k$, then the derived
simplicity of $A$ does not depend on the choice of the derived
category.
\end{cor}

In Section~\ref{ss:comm-ring} we will provide new
examples of $\cd(\Mod)$-simple algebras.

\subsection{$\cd^-(\Mod)$-simplicity}  This means that all ladders have height at most one:

\begin{thm}\label{t:dminus-simple-and-ladder} Let $A$ be a
$k$-algebra. Consider the following conditions: \begin{itemize}
\item[(i)] $A$ is $\cd^-(\Mod)$-simple,
\item[(ii)] there are no non-trivial perfect recollements of the form~(\ref{eq:recollement-unbounded-in-section-derived-simplicity}),
\item[(iii)] every non-trivial ladder of $\cd(\Mod A)$ has height
$\leq 1$.
\end{itemize}
Then (i)$\Rightarrow$(ii)$\Leftrightarrow$(iii). If $k$ is a field
and $A$ is finite-dimensional over $k$, then all three conditions
are equivalent.
\end{thm}
\begin{proof}
(i)$\Rightarrow$(ii) Suppose that $A$ has a recollement of the form
(\ref{eq:recollement-unbounded-in-section-derived-simplicity}) with
$i_*(B)\in K^b(\proj A)$. By
Proposition~\ref{p:recollement-restricting-to-d-Mod}, this
recollement restricts to $\cd^-(\Mod)$, yielding a non-trivial
recollement on the $\cd^-(\Mod)$-level and implying that $A$ is not
$\cd^-(\Mod)$-simple.

(ii)$\Rightarrow$(iii) Suppose that $\cl$ is a non-trivial ladder of
$\cd(\Mod A)$ of height $2$. By Lemma~\ref{l:ladder-of-height-2},
the upper recollement of $\cl$ is a non-trivial perfect recollement.

(iii)$\Rightarrow$(ii) Suppose
there is a non-trivial recollement of the form
(\ref{eq:recollement-unbounded-in-section-derived-simplicity}) with
$i_*(B)\in K^b(\proj A)$. By Proposition~\ref{p:ladder-extension}(a), this recollement, viewed as a ladder of height $1$, can be
extended downwards by one step, yielding a non-trivial ladder of
height $2$.

(ii)$\Rightarrow$(i) Let $k$ be a field and $A$ be a
finite-dimensional $k$-algebra. Suppose that $A$ is not
$\cd^-(\Mod)$-simple, i.e., there is a non-trivial recollement of
$A$ on the $\cd^-(\Mod)$-level. By
Proposition~\ref{p:lifting-recollement}, there is a non-trivial
recollement of the
form~(\ref{eq:recollement-unbounded-in-section-derived-simplicity})
which restricts to $\cd^-(\Mod)$. By
Proposition~\ref{p:recollement-restricting-to-d-Mod}, $i_*(B)\in
K^b(\proj A)$, i.e. this recollement is a non-trivial perfect
recollement.
\end{proof}

For general algebras, it is not the case that any recollement of the
form (\ref{eq:recollement-unbounded-in-section-derived-simplicity})
which can be restricted to $\cd^-(\Mod)$ is part of a ladder of height
$2$, as shown in the next example.

\begin{ex}\emph{(Example \ref{ex:large-tilting} continued.) }\label{ex:large-tilting-continued}{\rm Since $j_!(k)=eA\not\in \cd_{fl}(A)$,
it follows from Lemma~\ref{l:restriction-of-right-adjoint} and Lemma~\ref{l:useful-info}(e) that
$j_!$ does not admit a left adjoint, and hence the recollement
cannot be extended upwards
(Proposition~\ref{p:ladder-extension}(b)). Moreover, $j^*(A)\not\in
K^b(\proj k)$, equivalently, $i_*(A/AeA)\not\in K^b(\proj A)$,
implying that the recollement cannot be extended downwards
(Proposition~\ref{p:ladder-extension}(a)). So the recollement is a
complete ladder of height $1$. However, $AeA$ is projective, so it
follows by the short exact sequence \[\xymatrix{0\ar[r]
&AeA\ar[r]&A\ar[r]&A/AeA\ar[r] &0}\] that $i_*(A/AeA)\in K^b(\Proj
A)$, which implies that the recollement restricts to $\cd^-(\Mod)$
(Proposition~\ref{p:recollement-restricting-to-d-Mod}). In fact, the recollement restricts further to $\cd^b(\Mod)$ (Corollary~\ref{c:restricting-to-dbMod-when-C-has-finite-global-dimension}).}\end{ex}

It is proved in~\cite{LY} that finite-dimensional symmetric algebras
over a field do not admit any non-trivial perfect recollement, see~\cite[Remark 4.3]{LY}. Hence:

\begin{thm}
Let $k$ be a field and $A$ be a connected (i.e. indecomposable as an algebra) finite-dimensional symmetric
$k$-algebra. Then $A$ is $\cd^-(\Mod)$-simple.
\end{thm}

Below we construct a $\cd^-(\Mod)$-simple finite-dimensional
algebra which is not $\cd(\Mod)$-simple.

\begin{ex}\label{ex:Dminussimple-neq-Dsimple}{\rm Let $k$ be a field and let
$A$ be the $k$-algebra given by the quiver with relations
\[\xymatrix{1\ar@(ul,dl)_{\alpha}\ar@<.7ex>[r]^{\gamma} & 2\ar@<.7ex>[l]^{\beta}\ar@(ur,dr)^{\delta}},~~\beta\gamma\beta,\alpha^2,\gamma\alpha,\delta^2,\delta\gamma.\]
This algebra is 14-dimensional and the composition series of the two indecomposable projectives $P_1$ and $P_2$ are depicted as follows
\[\begin{smallmatrix} &&& 1 &&&\\[3pt] & 1 &&&& 2 &\\[3pt] & 2 &&& 1 && 2\\[3pt] 1 & & 2 &&&&\end{smallmatrix},\qquad
\begin{smallmatrix} &&& 2 &&\\[3pt] & 1 &&&& 2\\[3pt] &2 &&&&\\[3pt] 1 & & 2&&&\end{smallmatrix}~~.\]
We claim that, up to shift and isomorphism, $P_1$ and $P_2$ are the only indecomposable compact exceptional objects in $\cd(\Mod A)$.
Indeed,
let $X$ be an indecomposable
exceptional object of $K^b(\proj A)$, minimal of the form
\[\xymatrix{P_1^{\oplus m_a}\oplus P_2^{\oplus n_a}\ar[r]& P_1^{\oplus m_{a+1}}\oplus P_2^{\oplus n_{a+1}}\ar[r] &\hdots\ar[r]&P_1^{\oplus m_b}\oplus P_2^{\oplus n_b}},\]
where one of $m_a$ and $n_a$ is nonzero and one of $m_b$ and $n_b$ is nonzero.
Assume that $m_a\neq 0$ and $n_b\neq 0$: the other cases can be treated similarly. Consider the morphism $f:P_2\stackrel{\alpha\beta\delta}{\longrightarrow}P_1$, which maps the top of $P_2$ to the last radical layer of $P_1$. The map  $g:P_2^{\oplus n_b}\rightarrow P_1^{\oplus m_a}$ which, in matrix form, has all entries $f$, induces a self-extension of $X$ in degree $b-a$:
\[\xymatrix{\hdots\ar[r]&P_1^{\oplus m_{b-1}}\oplus P_2^{\oplus n_{b-1}}\ar[r]&P_1^{\oplus m_b}\oplus P_2^{\oplus n_b}\ar[d]^{\left(\begin{smallmatrix} 0 & g\\ 0 & 0\end{smallmatrix}\right)}\\
&&P_1^{\oplus m_a}\oplus P_2^{\oplus n_a}\ar[r]& P_1^{\oplus m_{a+1}}\oplus P_2^{\oplus n_{a+1}}\ar[r] &\hdots}.\]
The object $X$ being exceptional implies that $a=b$, { and further, either $m_a=0$ and $n_a=1$, or vice versa.}

 Next we show that each of $P_1$ and $P_2$ generates a recollement.
Consider the case for $P_1=e_1A$. As a right $A$-module the quotient $A/Ae_1A$ admits the following projective resolution
\[\xymatrix@C=1.5pc{\ldots\ar[r] & P_1\ar[r]^\alpha&P_1\ar[r]^\alpha&P_1\ar[r]^\gamma&P_2\ar[r] & A/Ae_1A\ar[r] & 0}.\]
Thus by Lemma~\ref{l:criterion-stratifying-ideal}, $Ae_1A$ is a stratifying ideal, and hence $P_1$ generates a  recollement of $\cd(\Mod A)$ by $\cd(\Mod A/Ae_1A)\cong\cd(\Mod k[x]/(x^2))$ and
$\cd(\Mod e_1Ae_1)\cong \cd(\Mod k\langle x,y\rangle/(x^2,y^2,xy))$ ($e_1Ae_1=k\{e_1,\beta\gamma,\alpha,\alpha\beta\gamma\}$).

The case for $P_2$ is similar: it generates a  recollement of $\cd(\Mod A)$ by $\cd(\Mod A/Ae_2A)\cong\cd(\Mod k[x]/(x^2))$ and $\cd(\Mod e_2Ae_2)\cong
\cd(\Mod k\langle x,y\rangle/(x^2,y^2,xy))$ ($e_2Ae_2=k\{e_2,\gamma\beta,\delta,\gamma\beta\delta\}$).

In particular, $A$ is not $\cd(\Mod)$-simple.
However, the two recollements are not in the same ladder, and hence both recollements are already complete ladders.  So all non-trivial ladders of $\cd(\Mod A)$ have height $1$. By Theorem~\ref{t:dminus-simple-and-ladder}, $A$ is $\cd^-(\Mod)$-simple.
}\end{ex}

\subsection{$K^b(\proj)$-simplicity} This means that all ladders have height at most two:

\begin{thm}\label{t:kbproj-simplicity-and-ladder}
 Let $A$ be a $k$-algebra. The following are equivalent:
\begin{itemize}
 \item[(i)] $A$ is $K^b(\proj)$-simple,
 \item[(ii)] all non-trivial ladders of $\cd(\Mod A)$ have height $\leq 2$.
\end{itemize}
\end{thm}
\begin{proof}
 (i)$\Rightarrow$(ii) Let $\cl$ be a non-trivial ladder of $\cd(\Mod A)$ of height $3$.
Then by Proposition~\ref{p:ladder-of-height-3}, the upper recollement of $\cl$ restricts to a non-trivial
recollement on the level of $K^b(\proj)$. Thus $A$ is not $K^b(\proj)$-simple.

(ii)$\Rightarrow$(i) Suppose that $A$ is not $K^b(\proj)$-simple. Then there is a non-trivial recollement
of $A$ on the $K^b(\proj)$-level. By Proposition~\ref{p:lifting-recollement}, there is a non-trivial recollement
of the form (\ref{eq:recollement-unbounded-in-section-derived-simplicity}) which restricts to $K^b(\proj)$.
Extending twice downwards by Proposition~\ref{p:ladder-extension}(a), we obtain a non-trivial ladder of height $3$.
\end{proof}

 Next we are going to provide an example of a
finite-dimensional algebra which is $K^b(\proj)$-simple but not
$\cd^-(\Mod)$-simple. This will be done by showing that there is only one non-trivial ladder, which is of height $2$. In \cite{LiLiping13} Liping Li gives a class of $K^b(\proj)$-simple algebras, which contains our example.

\begin{ex}\label{e:three-recollements-and-derived-simple} {\rm
Let $k$ be a field and let $A$ be the radical square zero $k$-algebra whose quiver is
\[\xymatrix{1\ar@(ul,dl)_{\gamma} & 2\ar[l]_{\alpha}\ar@(ur,dr)^{\beta}}\]
with indecomposable projective modules $P_1=\begin{smallmatrix} & 1 &\\ 1
& & 2\end{smallmatrix}$ and $P_2=\begin{smallmatrix} 2 \\2 \end{smallmatrix}\ $.
As in Example~\ref{e:three-recollements}, it can be checked that $P_1$, $P_2$, and
$X:=\mathrm{Cone}(P_2\stackrel{\alpha}{\rightarrow}P_1)$ are the only
indecomposable exceptional compact objects, up to shift and up to
isomorphism. We will show that $X$ does not generate a recollement of derived
module categories, while $P_1$ and $P_2$ are in the same ladder.

\smallskip

Consider the recollement of $\cd(\Mod A)$ generated by
$X=\Cone(P_2\ra P_1)$. The endomorphism ring of $X$ is
$\End_A(X)=k[x,y]/(x^2,y^2,xy)=:C$, where
\[x:\xymatrix{P_2\ar[r]^{\alpha}\ar[d]^\beta & P_1\ar[d]^0\\ P_2\ar[r]^{\alpha} & P_1},\qquad y:\xymatrix{P_2\ar[r]^{\alpha}\ar[d]^0 & P_1\ar[d]^\gamma\\ P_2\ar[r]^{\alpha} & P_1.}\] As a complex of left
$C$-modules, $X$ is isomorphic to $k[x]/x^2\xrightarrow{\alpha} \underline{k
\oplus k[y]/y^2}$, where the underlined term is in degree $0$, $k[x]/x^2$ and $k[y]/y^2$ are identified as
quotients of $C$ by factoring out the ideal generated by $y$ and
$x$, respectively, and $\alpha$ is the projection onto the trivial
$C$-module $k$. This complex splits into the direct sum of $k[1]$ and
$k[y]/y^2$. It is straightforward to show that as a complex of right $C$-modules
$X^{\tr}:=\Hom_A(X,A)$ is isomorphic to $\underline{k[y]/y^2} \ra k\oplus
k[x]/x^2$, which splits into the direct sum of $k$ and
$(k[x]/x^2)[1]$. Hence the total cohomology of $X^{\tr}\lten_C X$ is
infinite-dimensional. In particular it does not belong to
$\cd^b(\mod A)$. By Lemma \ref{l:finite-sides}, the recollement
 \[\xymatrix@!=7pc{**[r]\Dcal' \ar[r]|{i_*=i_!} &\cd(\Mod A) \ar@<+2.5ex>[l]|{i^!}
\ar@<-2.5ex>[l]|{i^*} \ar[r]|{j^!=j^*} &
\cd(\Mod C)\ar@<+2.5ex>[l]|{j_*} \ar@<-2.5ex>[l]|{j_!}}
\] generated by $X$ cannot be a recollement of ordinary algebras. More
precisely the right perpendicular category $\Dcal'$ of $X$ is not a
derived category of any ordinary algebra.

Consider the recollement generated by
$P_1=e_1A$, whose endomorphism ring $\End_A(P_1)=e_1Ae_1$ is
isomorphic to $k[x]/x^2$. As a right $A$-module $A/Ae_1A$ is isomorphic to $P_2$.
Thus it follows from Lemma~\ref{l:criterion-stratifying-ideal} that $P_1$ generates
a recollement of $\cd(\Mod A)$ by $\cd(\Mod A/Ae_1A)\cong\cd(k[x]/x^2)$ and
$\cd(\Mod e_1Ae_1)\cong\cd(k[x]/x^2)$. The corresponding TTF triple is $(\Tria(P_1),\Tria(P_2),\Tria(P_2)^\perp)$.

Consider the recollement generated by
$P_2=e_2A$. As a right $A$-module $A/Ae_2A$ admits the following projective resolution
\[\xymatrix{\ldots\ar[r]&P_2\ar[r]^\beta &P_2\ar[r]^\beta&P_2\ar[r]^\alpha &P_1\ar[r] &A/Ae_2A\ar[r] & 0}.\]
Thus it follows from Lemma~\ref{l:criterion-stratifying-ideal} that $P_2$ generates a
recollement of $\cd(\Mod A)$ by $\cd(\Mod A/Ae_2A)\cong\cd(k[x]/x^2)$ and
$\cd(\Mod e_2Ae_2)\cong\cd(k[x]/x^2)$. The  TTF triple corresponding to this recollement is $(\Tria(P_2),\Tria(A/Ae_2A),\Tria(A/Ae_2A)^\perp)$.

Clearly the above two TTF triples together form one TTF quadruple
\[(\Tria(P_1),\Tria(P_2),\Tria(A/Ae_2A),\Tria(A/Ae_2A)^\perp),\]
which is complete. Since this is the unique non-trivial TTF tuple of $\cd(\Mod A)$, it follows
that $A$ is $K^b(\proj)$-simple but not $\cd^-(\Mod)$-simple.
}
\end{ex}

\subsection{$\cd^b(\Mod)$-simplicity and $\cd_{fl}$-simplicity}

The following proposition follows immediately from Proposition~\ref{p:ladder-of-height-3}.

\begin{proposition}\label{p:dbMod-simplicity-and-ladder}
 Let $A$ be a $k$-algebra. If $A$ is $\cd^b(\Mod)$-simple or $\cd_{fl}$-simple, then
all non-trivial ladders of $\cd(\Mod A)$ have height $\leq 2$.
\end{proposition}

\begin{thm}\label{t:dbmod-simplicity-and-ladder}
Let $k$ be a field and $A$ be finite-dimensional over $k$. The following are equivalent:
\begin{itemize}
\item[(i)] $A$ is $\cd^b(\Mod)$-simple,
\item[(ii)] $A$ is $\cd^b(\mod)$-simple,
\item[(iii)] $A$ is $K^b(\proj)$-simple,
\item[(iv)] all non-trivial ladders of $\cd(\Mod A)$ have height $\leq 2$.
\end{itemize}
\end{thm}
\begin{proof}
Recall that in this case $\cd_{fl}(A)=\cd^b(\mod A)$.

(i)$\Leftrightarrow$(ii) This follows from Proposition~\ref{p:lifting-recollement} and Corollary~\ref{c:restricting-is-equivalent-for-dbMod-and-dbmod-for-fd-alg}.

(iii)$\Leftrightarrow$(iv) This is Theorem~\ref{t:kbproj-simplicity-and-ladder}.

(i)$\Rightarrow$(iv) This follows from Proposition~\ref{p:dbMod-simplicity-and-ladder}.

(iv)$\Rightarrow$(ii) Suppose that $A$ is not $\cd^b(\mod)$-simple. Then there is a non-trivial recollement
of $A$ on  $\cd^b(\mod)$-level. By Proposition~\ref{p:lifting-recollement}, this lifts to a recollement
of the form (\ref{eq:recollement-unbounded-in-section-derived-simplicity}) where $j_!$ restricts to
$\cd^b(\mod)$ and $i_*(B)\in K^b(\proj A)$, see Theorem~\ref{t:restricting-recollements-to-dbmod}. But then this recollement
of $\cd(\Mod A)$ extends to a non-trivial ladder of height $3$ by Proposition~\ref{p:ladder-extension}(a) and (b).
\end{proof}

\subsection{Indecomposable commutative rings are derived simple}\label{ss:comm-ring}
Let $A$ be a commutative ring. Given $\mathfrak p\in \mathrm{Spec}A$
and a complex of $A$-modules $$X:\; \cdots\to
X^i\stackrel{d^i}{\to}X^{i+1}\to\cdots$$ we consider the complex
$$X_{\mathfrak p}:\; \cdots \to X^i\otimes_A A_{\mathfrak
p}\stackrel{d^i\otimes_A A_{\mathfrak p}}{\longrightarrow}
X^{i+1}\otimes_A A_{\mathfrak p}\to\cdots$$  Since $?\otimes_A
A_{\mathfrak p}$ is an exact functor, we have $H^i(X_{\mathfrak
p})=H^i(X)\otimes_A A_{\mathfrak p}$.

\bigskip

\begin{lemma}\label{formula}
For $X\in K^b(\proj A)$, $Y\in \cd({\Mod}A)$, $\mathfrak
p\in\mathrm{Spec}A$ and any integer $n$, there is an isomorphism
$$\Hom_{\cd({\Mod}A)}(X,Y[n])\otimes_A A_{\mathfrak
p}\cong\Hom_{\cd({\Mod}A_{\mathfrak p})}(X_{\mathfrak
p},Y_{\mathfrak p}[n])$$
\end{lemma}
\begin{proof}
$\Hom_{\cd({\Mod}A)}(X,Y[n])=\Hom_{K({\Mod}A)}(X,Y[n])$
is the $n$-th cohomology of the total complex ${\mathcal
Hom}_A(X,Y)$. The well-known formula \cite[3.2.4]{EJ}
$$\Hom_{A}(X^i,Y^j)\otimes_A A_{\mathfrak p}\cong\Hom_{A_{\mathfrak
p}}(X_{\mathfrak p}^i,Y_{\mathfrak p}^j)$$ implies that ${\mathcal
Hom}_A(X,Y)_{\mathfrak p}\cong{\mathcal Hom}_{A_{\mathfrak
p}}(X_{\mathfrak p},Y_{\mathfrak p})$. The claim now follows from
the fact that localisation preserves cohomologies  and $X_{\mathfrak
p}\in K^b(P_{A_{\mathfrak p}})$.
\end{proof}

\begin{prop} \label{indec-compact-except}
Let $X$ be an compact exceptional object in $\cd({\Mod}A)$. If $X$ is indecomposable, then $X$ is of the form $P[n]$ with $P$ finitely generated projective and $n\in\Z$.
\end{prop}
\begin{proof}
Assume on the contrary that $X$ has a $K^b(\proj A)$-representative of the
form $$\ldots \ra 0 \to P^{-n}\to\cdots\to P^{-1} \xrightarrow{d^{-1}} P^{0}\to 0\ra \ldots$$ with $n\geq 1$ minimal. Then $C=H^0(X)$ is nontrivial and moreover the projective dimension of $C$ must be at least one. In fact, if $C$ were projective, $\Img(d^{-1})$ would be a common direct summand of $P^0$ and $P^{-1}$, and $X$ would be the direct sum of $C$ and the complex $$\ldots \ra 0 \to P^{-n}\to\cdots\to P^{-2} \ra \Ker(d^{-1})\to 0 \ra \ldots$$ contradicting the assumption that $X$ is indecomposable. Now by Lemma \ref{formula} there is $\mathfrak{p}\in \text{Spec}A$ such that $C_{\mathfrak p}$ has the same projective dimension as $C$. Since $C_{\mathfrak p} = H^0(X)\otimes_A A_{\mathfrak p} \cong H^0(X_{\mathfrak p})$, the complex $X_{\mathfrak p}$ cannot be isomorphic to a shifted projective module. As in \cite[4.9]{AKL2} we conclude that $X_{\mathfrak p}$ is not exceptional. But then $X$ cannot be exceptional, by again Lemma \ref{formula}.
\end{proof}

 The Proposition shows that every compact exceptional object in $\cd(\mathrm{Mod}A)$ is a direct sum of shifted projective modules. In particular:

\begin{thm} \label{thm:derivedsimplicity-commutativering} Every indecomposable commutative ring
$A$ is derived simple with respect to $\cd(\Mod)$.
\end{thm}


\begin{proof} Let $A$ be indecomposable and $P$ a finitely generated projective $A$-module. By \cite[2.44]{Lam}, the trace $\tau_P(A)$ of $P$ in $A$ is  a direct summand of $A$, thus either zero or equal to $A$. This implies that $P$ is zero or a projective generator of $\Mod A$.
As recalled in \ref{ss:recollement-of-der-cat}, every recollement of the form
(\ref{eq:recollement-unbounded-der}) is generated by a compact exceptional object $X$ and the right hand side in (\ref{eq:recollement-unbounded-der}) equals $\Tria X$. We claim that $X$ must be a shifted projective module. If not, by Proposition \ref{indec-compact-except} it has at least two direct summands, say $P[n]$ and $Q[m]$ with $n\neq m$. Since $P$ and $Q$ are both projective generators of $\Mod A$ we have $\Hom_A(P,Q)\neq 0$ and thus $\Hom_{\Dcal(\Mod A)}(P[n],Q[m][n-m])\cong \Hom_{\Dcal(\Mod A)}(P[n],Q[n])\neq 0$. This shows that $P[n]\oplus Q[m]$ cannot be exceptional. Now $X$ has the form $P[n]$ with $P$ finitely generated projective. Then $\Tria X = \Tria P[n] = \Tria P = \cd(\Mod A)$. The claim is proven.
\end{proof}

Moreover, we recover a result from \cite{CM,PT} as a special case.

\begin{cor} Every finitely generated tilting module over a commutative ring is projective.\end{cor}
\begin{proof}
It is well known that every finitely generated tilting  module $T$
has a projective resolution with finitely generated projective
modules (e.g.~by combining the fact that every tilting class is definable
\cite{BS} with  \cite[9.13 (5)]{relml}).  Then $T$ is a compact
exceptional object, so it is projective.
\end{proof}

Here is another consequence of the Proposition above.

\begin{cor}
A commutative ring $A$ is derived equivalent to a ring $B$ if and only if $A$ and $B$ are Morita equivalent.
\end{cor}
\begin{proof}
By a well known result due to Rickard \cite{Rickard89},    $A$ and $B$ are derived equivalent if and only if $B$ is the endomorphism ring of a tilting complex $T$ over $A$. But, as  shown above,  $T$ is of the form $P[n]$ where $P$ is a finitely generated projective generator of $\Mod A$  and $n\in\Z$. Hence $B\cong \End_A(P)$ is Morita equivalent to $A$.
\end{proof}

As a consequence, the derived Picard group is the direct product of the infinite cyclic group generated by the shift in the derived category and the classical Picard group, which describes the Morita equivalences.

\smallskip

Finally we briefly mention a result on co-t-structures and leave the details to the interested readers. Let $A$ be an indecomposable commutative ring. As in the proof of Theorem~\ref{thm:derivedsimplicity-commutativering}, one can show that any presilting object of $K^b(\proj A)$ is either $0$ or of the form $P[n]$ for some integer $n$ and some finitely generated projective generator $P$ of $\Mod A$. This provides an alternative approach and generalises \cite[5.7]{StovicekPospisil12}. Moreover, since
any object of the co-heart of a co-$t$-structure on $K^b(\proj A)$ is presilting, it follows that if the co-heart is non-trivial, then it is $(\proj A)[n]$ for some integer $n$, in particular, the co-$t$-structure is bounded. Consequently,
by \cite[5.9]{MendozaSaenzSantiagoSouto10}, any  co-$t$-structure on $K^b(\proj A)$ with non-trivial co-heart is a shift of the standard one, compare \cite[5.6]{StovicekPospisil12}.

\section{Algebraic K-theory}\label{s:k-theory}

This section is devoted to the study of K-groups. { We restrict to finite-dimensional algebras.  In this case} Schlichting's $(-1)$-st K-group vanishes. Using this result and a result on silting objects we show that if a $\cd(\Mod)$-recollement is given, the Grothendieck group of the middle algebra is the direct sum of those of the two outer algebras. For a $\cd^-(\Mod)$-recollement, we use a method of Chen and Xi to show that the $i$-th K-group is the direct sum of the $i$-th K-groups of the two outer algebras for any integer $i$.

\subsection{Frobenius pairs and K-groups}
We follow~\cite{Schlichting06}. A \emph{Frobenius pair} is a pair $(\ca,\ca_0)$, where $\ca$ is a small Frobenius category and $\ca_0$ is a full Frobenius subcategory, { i.e. an extension closed full subcategory of $\ca$ which inherits the structure of a Frobenius category}. To a Frobenius pair $(\ca,\ca_0)$ and $i\in\mathbb{Z}$ we associate the $i$-th K-group $\mathbb{K}_i(\ca,\ca_0)$ of $(\ca,\ca_0)$, {see~\cite[Section 12]{Schlichting06}}. For $i=0$, the group $\mathbb{K}_0(\ca,\ca_0)$ is the Grothendieck group of the idempotent completion {(\cite[Definition 1.2]{BalmerSchlichting01})} of the associated triangulated category $\cd(\ca,\ca_0)=\underline{\ca}/\underline{\ca_0}$.

{ For example, for a finite-dimensional algebra $A$ over a field $k$, let $\cc^b(\mod A)$ be the category of bounded complexes of $A$-modules from $\mod A$ and $\acyc^b(\mod A)$ be its full subcategory of acyclic complexes. Then $\mathbb{K}_0(\cc^b(\mod A),\acyc^b(\mod A))=\mathbb{K}_0(\cd^b(\mod A))$, which is isomorphic to the usual Grothendieck group $\mathbb{K}_0(A)$ of the algebra $A$, cf.\cite[III.1]{Happel}.
}

A functor of Frobenius pairs $(\ca,\ca_0)\rightarrow (\cb,\cb_0)$ is a functor $\ca\rightarrow\cb$ of Frobenius categories which restricts to a functor $\ca_0\rightarrow\cb_0$. A sequence $(\ca,\ca_0)\rightarrow(\cb,\cb_0)\rightarrow (\cc,\cc_0)$ of Frobenius pairs is a \emph{short exact sequence} if the induced sequence of triangulated categories $\cd(\ca,\ca_0)\stackrel{i}{\rightarrow}\cd(\cb,\cb_0)\stackrel{p}{\rightarrow} \cd(\cc,\cc_0)$ is short exact, i.e. $i$ is fully faithful and $p$ induces an equivalence $\cd(\cb,\cb_0)/\im(i)\stackrel{\sim}{\rightarrow}\cd(\cc,\cc_0)$ up to direct summands.

\begin{thm}\emph{(\cite[Theorem 9]{Schlichting06})}\label{t:long-exact-seq-of-k-groups} Let $(\ca,\ca_0)\rightarrow(\cb,\cb_0)\rightarrow (\cc,\cc_0)$ be a short exact sequence of Frobenius pairs. Then there is a long exact sequence of K-groups
\[\xymatrix@R=0.5pc@C=0.9pc{\ldots\ar[r] & \mathbb{K}_i(\ca,\ca_0)\ar[r] &\mathbb{K}_i(\cb,\cb_0)\ar[r] & \mathbb{K}_i(\cc,\cc_0)
\ar[r] &\mathbb{K}_{i-1}(\ca,\ca_0)\ar[r] &\mathbb{K}_{i-1}(\cb,\cb_0)\ar[r] & \mathbb{K}_{i-1}(\cc,\cc_0)\ar[r] &\ldots}\]
\end{thm}

{\it From now on, let $k$ be a field and let $A$ be a finite-dimensional $k$-algebra.} Recall that $\proj A$ denotes the category of finitely generated projective $A$-modules. By abuse of notation, we will also denote by $\proj A$ its skeleton, which by definition consists of one representative from each isomorphism class of objects. Let $\cc^b(\proj A)$ be the category of bounded complexes of finitely generated projective $A$-modules. It has a natural structure of a Frobenius category: the conflations are the componentwise split short exact sequence of complexes. Let $\cc^b_0(\proj A)$ be the full subcategory of $\cc^b(\proj A)$ consisting of null-homotopic complexes. Then $(\cc^b(\proj A),\cc^b_0(\proj A))$ is a Frobenius pair and {we call it a \emph{Frobenius model} of $K^b(\proj A)$ since the associated triangulated category $\underline{\cc^b(\proj A)}/\underline{\cc^b_0(\proj A)}$ is $K^b(\proj A)$}. We denote $\mathbb{K}_i(A)=\mathbb{K}_i(\cc^b(\proj A),\cc^b_0(\proj A))$.

\subsection{Vanishing of $\mathbb{K}_{-1}$}  Recall that $\mod A$ denotes the category of finitely generated $A$-modules. By abuse of notation, we will also denote by $\mod A$ its skeleton. Define the singularity category $\cd_{sg}(A)$ as the triangle quotient $\cd^b(\mod A)/K^b(\proj A)$.

Let $\cc^b_0(\mod A)$ be the Frobenius subcategory of $\cc^b(\mod A)$ corresponding to the essential image of the embedding $K^b(\proj A)\hookrightarrow \cd^b(\mod A)$. Then $(\cc^b(\mod A),\acyc^b(\mod A))$ and $(\cc^b(\mod A),\cc^b_0(\mod A))$ are respectively Frobenius models of $\cd^b(\mod A)$ and $\cd_{sg}(A)$. Therefore we have a short exact sequence of Frobenius pairs
\[\xymatrix{(\cc^b(\proj A),\cc^b_0(\proj A))\ar[r] & (\cc^b(\mod A),\acyc^b(\mod A))\ar[r] &(\cc^b(\mod A),\cc^b_0(\mod A))},\]
which induces a long exact sequence
\[\xymatrix@R=0.4pc@C=0.9pc{\ldots\ar[r] & \mathbb{K}_0(A)\ar[r] &\mathbb{K}_0(\cc^b(\mod A),\acyc^b(\mod A))\ar[r] & \mathbb{K}_0(\cc^b(\mod A),\cc^b_0(\mod A))\\
\ar[r] &\mathbb{K}_{-1}(A)\ar[r] &\mathbb{K}_{-1}(\cc^b(\mod A),\acyc^b(\mod A))\ar[r] & \mathbb{K}_{-1}(\cc^b(\mod A),\cc^b_0(\mod A))\ar[r] &\ldots}\]
By~\cite[Theorem 6]{Schlichting06}, $\mathbb{K}_{-1}(\cc^b(\mod A),\acyc^b(\mod A))$ vanishes. As a consequence, $\mathbb{K}_{-1}(A)$ is exactly the obstruction of the idempotent completeness of $\cd_{sg}(A)$.

\begin{proposition}\label{p:Kminus1-as-obstruction}
 $\mathbb{K}_{-1}(A)$ vanishes if and only if $\cd_{sg}(A)$ is idempotent complete.
\end{proposition}
\begin{proof}
We obtain from above an exact sequence
\[\xymatrix{\mathbb{K}_0(\cc^b(\mod A),\acyc^b(\mod A))\ar[r]^(0.52)p & \mathbb{K}_0(\cc^b(\mod A),\cc^b_0(\mod A))
\ar[r] &\mathbb{K}_{-1}(A)\ar[r] & 0}.\]
{Therefore  $\mathbb{K}_{-1}(A)=0$ holds  if and only if $p$ is surjective. The latter condition is satisfied if and only if $\cd_{sg}(A)$ is idempotent complete, see for example~\cite[Remark 1]{Schlichting06}. }
\end{proof}

\begin{corollary}\label{c:vanishing-of-Kminus1}  $\mathbb{K}_{-1}(A)=0$.
\end{corollary}
\begin{proof} By~\cite[Corollary 2.4]{ChenXW11}, $\cd_{sg}(A)$ is idempotent complete. The desired result follows immediately from Proposition~\ref{p:Kminus1-as-obstruction}.
\end{proof}

\subsection{The long exact sequence}\label{ss:long-exact-sequence}
{ For our purpose it will be useful to employ} another Frobenius model for $K^b(\proj A)$.
Let $\cc^{-,c}(\proj A)$ denote the category of right bounded complexes which are homotopic equivalent to complexes in $\cc^b(\proj A)$ and let $\cc^{-,c}_0(\proj A)$ denote its full subcategory consisting of null-homotopic complexes. Let $K^{-,c}(\proj A)$ denote the stable category of $\cc^{-,c}(\proj A)$. Then the canonical embedding $(\cc^b(\proj A),\cc^b_0(\proj A))\rightarrow(\cc^{-,c}(\proj A),\cc^{-,c}_0(\proj A))$ induces a triangle equivalence $K^b(\proj A)\rightarrow K^{-,c}(\proj A)$, and by~\cite[Theorem 9]{Schlichting06}, we have a canonical isomorphism for $i\in\mathbb{Z}$
$$\mathbb{K}_i(A)=\mathbb{K}_i(\cc^b(\proj A),\cc^b_0(\proj A))\cong \mathbb{K}_i(\cc^{-,c}(\proj A),\cc^{-,c}_0(\proj A))$$

{
Assume that there is a recollement of the form (\ref{eq:recollement-unbounded-der}). Recall from Lemma~\ref{l:finite-sides} and Lemma~\ref{l:useful-info} that $B$ and $C$ are necessarily finite-dimensional over $k$,    the functors $i^*$ and $j_!$ restrict to $K^b(\proj)$, and moreover,
 there is a right bounded complex of finitely generated projective $C$-$A$-bimodules $X$ such that $j_!=?\lten_C X=?\ten_C X$. As a complex of $A$-modules, $X$ belongs to $\cc^{-,c}(\proj A)$ and it follows that $?\ten_C X: \cc^{-,c}(\proj C)\rightarrow \cc^{-,c}(\proj A)$ is a well-defined functor
 of Frobenius categories, which induces a functor $(\cc^{-,c}(\proj C),\cc^{-,c}_0(\proj C))\rightarrow (\cc^{-,c}(\proj A),\cc^{-,c}_0(\proj A))$ of Frobenius pairs. Similarly, there is a right bounded complex of finitely generated $A$-$B$-bimodules $Y$ such that $i^*=?\ten_A Y$ induces a functor $(\cc^{-,c}(\proj A),\cc^{-,c}_0(\proj A))\rightarrow (\cc^{-,c}(\proj B),\cc^{-,c}_0(\proj B))$ of Frobenius pairs.
So, we obtain a sequence of Frobenius pairs
\[\xymatrix{(\cc^{-,c}(\proj C),\cc^{-,c}_0(\proj C))\ar[r]& (\cc^{-,c}(\proj A),\cc^{-,c}_0(\proj A))\ar[r] & (\cc^{-,c}(\proj B),\cc^{-,c}_0(\proj B))}.\]
We claim that this is a  short exact  sequence.
In fact, by~\cite[Theorem 2.1]{Neeman92a}, there is an equivalence of triangulated categories up to direct summands
\begin{eqnarray}K^b(\proj A)/\tria(j_!(C))\stackrel{\simeq}{\longrightarrow}K^b(\proj B)\label{eq:first-row-compact}\end{eqnarray}
which is even an equivalence, because $\mathbb{K}_{-1}(C)=0$ and $K^b(\proj A)/\tria(j_!(C))$ is thus idempotent complete by~\cite[Remark 1]{Schlichting06}.
}

It follows from~\cite[Theorem 9]{Schlichting06} that there is a long exact sequence for $i\in\mathbb{Z}$
\begin{eqnarray*}\xymatrix@R=0.5pc@C=1pc{\ldots\ar[r] & \mathbb{K}_i(C)\ar[r] &\mathbb{K}_i(A)\ar[r] & \mathbb{K}_i(B)
\ar[r] &\mathbb{K}_{i-1}(C)\ar[r] &\mathbb{K}_{i-1}(A)\ar[r] & \mathbb{K}_{i-1}(B)\ar[r] &\ldots}\end{eqnarray*}
By Corollary~\ref{c:vanishing-of-Kminus1}, the groups $\mathbb{K}_{-1}(A)$, $\mathbb{K}_{-1}(B)$ and $\mathbb{K}_{-1}(C)$ are trivial. Thus we have the following corollary.
\begin{corollary} Let $A$, $B$ and $C$ be finite-dimensional $k$-algebras admiting a recollement of the form (\ref{eq:recollement-unbounded-der}). Then there are  long exact sequences of K-groups
\[\xymatrix@R=0.5pc@C=1pc{\cdots\ar[r]&\mathbb{K}_i(C)\ar[r] &\mathbb{K}_i(A)\ar[r]&\mathbb{K}_i(B)\ar[r]&
\cdots\ar[r] & \mathbb{K}_0(C)\ar[r] & \mathbb{K}_0(A)\ar[r] & \mathbb{K}_0(B)\ar[r] & 0,}~~i\geq 0,\]
\[\xymatrix@R=0.5pc@C=1pc{0\ar[r] & \mathbb{K}_{-2}(C)\ar[r]&\mathbb{K}_{-2}(A)\ar[r] &\mathbb{K}_{-2}(B)\ar[r] &\ldots\ar[r] &\mathbb{K}_i(C)\ar[r] &\mathbb{K}_i(A)\ar[r]&\mathbb{K}_i(B)\ar[r]&
 \ldots,}~~i\leq -2.\]
\end{corollary}

\subsection{The Grothendieck group}\label{s:the-rank} Let $k$ be a field and let $A$, $B$ and $C$ be finite-dimensional $k$-algebras admitting a recollement of the form~(\ref{eq:recollement-unbounded-der}).  { We can assume that  $A$, $B$ and $C$ are basic.}

Denote by $r(A)$ the number of isomorphism classes of simple $A$-modules, which equals the rank of the Grothendieck group $\mathbb{K}_0(A)$.

\begin{proposition}\label{p:rank-and-recollement}
Let $A$, $B$ and $C$ be
basic
finite-dimensional $k$-algebras admiting a recollement of the form (\ref{eq:recollement-unbounded-der}). Then $r(A)=r(B)+r(C)$. In particular, there is a short exact sequence
\[\xymatrix{0\ar[r] & \mathbb{K}_0(C)\ar[r] & \mathbb{K}_0(A)\ar[r] & \mathbb{K}_0(B)\ar[r] & 0.}\]
\end{proposition}

{We need some preparation.}
An object $X$ of $K^b(\proj A)$ is called a \emph{presilting object} if $\Hom(X,\Sigma^iX)=0$ for any $i>0$ and a \emph{silting object} if in addition $K^b(\proj A)=\tria(X)$.
The first statement of the following proposition is a special case of \cite[Theorem 4.8]{IyamaYang12}. The second statement follows from \cite[Theorem 4.1]{IyamaYang12}.

\begin{proposition}\label{p:silting-reduction}
Let $X$ be a basic presilting object of $K^b(\proj A)$. Then the quotient functor $\pi:K^b(\proj A)\rightarrow K^b(\proj A)/\tria(X)$ induces a one-to-one correspondence between the set of isomorphism classes of basic silting objects in $K^b(\proj A)$ containing $X$ as a direct summand and the set of isomorphism classes of basic silting objects in $K^b(\proj A)/\tria(X)$. Let $T$ be an element of the former set; then $\End_{K^b(\proj A)/\tria(X)}(\pi(T))$ is isomorphic to the quotient of $\End_{K^b(\proj A)}(T)$ by the ideal generated by the idempotent corresponding to $X$.
\end{proposition}

\begin{proof}[Proof of Proposition~\ref{p:rank-and-recollement}]
Observe that $j_!(C)$ is a basic presilting object in $K^b(\proj A)$ (it is actually a partial tilting object). Let $\pi$ denote the composition of the quotient functor $\pi:K^b(\proj A)\rightarrow K^b(\proj A)/\tria(j_!(C))$ and the triangle equivalence (\ref{eq:first-row-compact}). By Proposition~\ref{p:silting-reduction}, there is a silting object $T$ in $K^b(\proj A)$ which contains $j_!(C)$ as a direct summand such that $\pi(T)=B$. Then $B=\End_{K^b(\proj B)}(B)=\End_{K^b(\proj A)}(T)/(e)$, where $e$ is the idempotent corresponding to the direct summand $X$ of $T$, and $(e)$ is the ideal generated by $e$. Hence $$r(B)=r(\End_{K^b(\proj A)}(T))-r(e\End_{K^b(\proj A)}(T)e)=r(\End_{K^b(\proj A)}(T))-r(C).$$ Since $r(\End_{K^b(\proj A)}(T))$ equals the number of indeomposable direct summands of $T$, which equals $r(A)$ by~\cite[Theorem 2.26]{AiharaIyama10}, it follows that
$r(A)=r(B)+r(C)$,
as desired.\end{proof}

\subsection{Decomposing K-groups along recollements}
The following theorem has been motivated and inspired by the results in \cite{ChenXi12a}, which it complements and strengthens in the case of finite-dimensional algebras. We are indebted to Changchang Xi for informing us about these results.

\begin{thm} \label{t:splitting-K-groups}
Let $A$, $B$ and $C$ be finite-dimensional $k$-algebras admitting a $\mathcal D^-({\rm Mod})$-recollement of the form (\ref{eq:recollement-bounded-above}). {Then there are isomorphisms $\mathbb{K}_i(A)\cong \mathbb{K}_i(B)\oplus \mathbb{K}_i(C)$ for $i\in\mathbb{Z}$}.
\end{thm}

\begin{proof}
By Proposition~\ref{p:lifting-recollement}, we may view the given recollement as a recollement restricted from a $\cd(\Mod)$-recollement.  As in Section~\ref{ss:long-exact-sequence}, consider the short exact sequence of Frobenius pairs
\[\xymatrix{(\cc^{-,c}(\proj C),\cc^{-,c}_0(\proj C))\ar[r]^\iota& (\cc^{-,c}(\proj A),\cc^{-,c}_0(\proj A))\ar[r]^\pi & (\cc^{-,c}(\proj B),\cc^{-,c}_0(\proj B))},\]
which induces a short exact sequence of triangulated categories (recall that the canonical embedding $K^b(\proj)\rightarrow K^{-,c}(\proj)$ is a triangle equivalence)
\[\xymatrix{K^{-,c}(\proj C)\ar[r]^{j_!} &K^{-,c}(\proj A)\ar[r]^{i^*}&K^{-,c}(\proj B)}\]
and a long exact sequence of K-groups
\begin{eqnarray}\label{eq:long-exact-sequence}\xymatrix@R=0.5pc@C=1.5pc{\ldots\ar[r] &\mathbb{K}_{i+1}(B)\ar[r]^{\delta_{i+1}}& \mathbb{K}_i(C)\ar[r]^{\mathbb{K}_i(\iota)} &\mathbb{K}_i(A)\ar[r]^{\mathbb{K}_i(\pi)} & \mathbb{K}_i(B)
\ar[r]^{\delta_i} &\mathbb{K}_{i-1}(C)\ar[r]^{\mathbb{K}_{i-1}(\iota)} &\mathbb{K}_{i-1}(A)\ar[r]&\ldots}\end{eqnarray}

Recall from Section~\ref{ss:long-exact-sequence} that there is a right bounded complex of finitely generated $C$-$A$-bimocules $X$ such that $j_!=?\lten_C X$. It follows that  $j^*=?\lten_A X^{\tr}$ (see Lemma~\ref{l:finite-sides}). By Proposition~\ref{p:recollement-restricting-to-d-Mod}, $i_*(B)\in K^b(\proj A)$ holds. It follows from Lemma~\ref{l:restriction-to-compact-for-row} that $j^*(A)=X^{\tr}\in K^b(\proj C)$. So there is a right bounded complex of finitely generated projective $A$-$C$-bimodules $X'$ which is quasi-isomorphic to $X^{\tr}$ as a complex of bimodules; so $j^*\simeq ?\lten_A X'=?\ten_A X'$. Moreover, as a complex of $C$-modules $X'$ belongs to $\cc^{-,c}(\proj C)$. Therefore $?\ten_A X'$ defines a functor of Frobenius pairs $\kappa:(\cc^{-,c}(\proj A),\cc^{-,c}_0(\proj A))\rightarrow(\cc^{-,c}(\proj C),\cc^{-,c}_0(\proj C))$. The composition $\kappa\iota:(\cc^{-,c}(\proj C),\cc^{-,c}_0)(\proj C)\rightarrow(\cc^{-,c}(\proj C),\cc^{-,c}_0)(\proj C)$ induces the triangle functor $j^*j_!:K^{-,c}(\proj C)\rightarrow K^{-,c}(\proj C)$, which is equivalent to the identity. Thus by~\cite[Theorem 9]{Schlichting06}, $\mathbb{K}_i(\kappa)\circ\mathbb{K}_i(\iota)=\mathbb{K}_i(\kappa\iota):\mathbb{K}_i(C)\rightarrow\mathbb{K}_i(C)$ is an isomorphism. In particular, {$\mathbb{K}_i(\iota)$ is a split momomorphism for any $i\in\mathbb{Z}$}. We are done.
\end{proof}

The following corollary is a direct consequence of Theorem~\ref{t:splitting-K-groups}.
\begin{corollary} Let $k$ be a field and $A$ be a finite-dimensional $k$-algebra admitting a $\cd^-(\Mod)$-stratification with simple factors $A_1,\ldots,A_s$. Then there are isomorphisms of K-groups $\mathbb{K}_i(A)\cong\bigoplus_{j=1}^s\mathbb{K}_i(A_j)$ for $i\in\mathbb{Z}$.
\end{corollary}

Prominent finite-dimensional algebras whose derived categories admit stratifications are quasi-hereditary algebras. Here, the factors in the stratification are derived categories of vector spaces over the endomorphism rings of the simple modules, i.e. over division rings. In particular, the K-theory of Schur algebras of algebraic groups and of blocks of the Bernstein--Gelfand--Gelfand category $\mathcal O$ of a semisimple complex Lie algebra
decomposes into a direct sum of as many copies of the K-theory of the ground field as there are simple modules. In a similar way, the K-theory of hereditary algebras and of
algebras of global dimension two can be decomposed.

\section{The derived Jordan--H\"older theorem}\label{JH}   When recollements of derived module
categories were studied first, around 1990, the question came up
whether a derived
Jordan--H\"older theorem holds true, that is, whether finite
stratifications exist and are unique in the sense that the simple
factors of any two stratifications (multiplicities counted) are the
same.

In this section, we  study only $\cd(\Mod)$-stratifications. Recently it has been shown that such a Jordan--H\"older theorem
holds true for hereditary artin algebras~\cite{AKL2}, for
piecewise hereditary algebras~\cite{AKL3}, and for symmetric
algebras satisfying certain homological condition~\cite{LY}.
It turns out, however, to be false in general, see~\cite{CX1}
and~\cite{CX2} for infinite-dimensional counterexamples.

\smallskip

We first show that, for algebras with a block decomposition, the validity of the
Jordan--H\"older  theorem reduces  to the \emph{blocks}, that is, to
the indecomposable ring direct summands. The number of blocks is a
derived invariant.  For a stratification $\cs$, we denote by
$\mathrm{SF}(\cs)$ the sequence of simple factors
of $\cs$. For two stratifications $\cs$ and $\cs'$, we say that
$\mathrm{SF}(\cs)$ and $\mathrm{SF}(\cs')$ are \emph{equivalent} and
denote by $\mathrm{SF}(\cs)\sim\mathrm{SF}(\cs')$ if the two sequences
$\mathrm{SF}(\cs)$ and $\mathrm{SF}(\cs')$ are the same up to
triangle equivalence and reordering of their elements.

\begin{lemma}\label{l:JH-for-derived-semisimple-alg}
Let $A$ be a $k$-algebra with a block decomposition
$A=A_1\oplus\ldots\oplus A_s$. Then the Jordan--H\"older theorem
holds true $A$ if and only if it holds true for each $A_i$ (for any
choice of derived category). Moreover, if $\cs,\cs_1,\ldots,\cs_s$
are stratifications of $\cd(\Mod A),\cd(\Mod
A_1),\ldots,\cd(\Mod A_s)$, respectively, then $\mathrm{SF}(\cs)$ is equivalent to
the sequence $(\mathrm{SF}(\cs_1),\ldots,\mathrm{SF}(\cs_s))$.
\end{lemma}
\begin{proof} We only prove the lemma for $\cd(\Mod)$; similar
arguments work for other derived categories.

 For each
$i=1,\ldots,s$, let $\cs_i$ be a stratification of $\cd(\Mod A_i)$.
Then we can `glue' these $\cs_i$'s to obtain a stratification $\cs$
of $\cd(\Mod A)$, which satisfies $\mathrm{SF}(\cs)=(\mathrm{SF}(\cs_1),\ldots,\mathrm{SF}(\cs_s))$.
Now fix any $i$ and let $\cs'_i$ be another
stratification of $\cd(\Mod A_i)$, then glueing
$\cs_1,\ldots,\cs'_i,\ldots,\cs_s$ we obtain a stratification $\cs'$
of $\cd(\Mod A)$ with $\mathrm{SF}(\cs')=(\mathrm{SF}(\cs_1),\ldots,\mathrm{SF}(\cs_{i-1}),\mathrm{SF}(\cs'_i),
\mathrm{SF}(\cs_{i+1}),\ldots,\mathrm{SF}(\cs_s))$. If the
Jordan--H\"older theorem holds for $A$, then
$\mathrm{SF}(\cs)\sim\mathrm{SF}(\cs')$. It follows that
$\mathrm{SF}(\cs_i)\sim\mathrm{SF}(\cs'_i)$, namely, the
Jordan--H\"older theorem holds for $A_i$.

Conversely, by
Lemma~\ref{l:commutating-direct-sum-decomposition-and-recollements},
a recollement
$$\xymatrix@!=6pc{\cd(\Mod B) \ar[r]|{i_*=i_!} &\cd(\Mod A) \ar@<+2.5ex>[l]|{i^!}
\ar@<-2.5ex>[l]|{i^*} \ar[r]|{j^!=j^*} & \cd(\Mod
C)\ar@<+2.5ex>[l]|{j_*} \ar@<-2.5ex>[l]|{j_!} }$$ restricts to
recollements
$$\xymatrix@!=6pc{\cd(\Mod B_i) \ar[r]|{i_*=i_!} &\cd(\Mod A_i) \ar@<+2.5ex>[l]|{i^!}
\ar@<-2.5ex>[l]|{i^*} \ar[r]|{j^!=j^*} & \cd(\Mod
C_i)\ar@<+2.5ex>[l]|{j_*} \ar@<-2.5ex>[l]|{j_!} }$$ such that
$B=\bigoplus_i B_i$ and $C=\bigoplus_i C_i$. Thus it follows that a
stratification $\cs$ of $\cd(\Mod A)$ can be glued from
stratifications $\cs_i$ of $\cd(\Mod A_i)$. In particular,
$\mathrm{SF}(\cs)=(\mathrm{SF}(\cs_1),\ldots,\mathrm{SF}(\cs_s))$. Let $\cs'$ be
another stratification for $\cd(\Mod A)$, then there are
stratifications $\cs'_i$ of $\cd(\Mod A_i)$ such that
$\mathrm{SF}(\cs')=(\mathrm{SF}(\cs'_1),\ldots,\mathrm{SF}(\cs'_s))$. If the
Jordan--H\"older theorem holds true for all $A_i$, then
$\mathrm{SF}(\cs_i)\sim\mathrm{SF}(\cs'_i)$ holds for all $i$. It
follows that $\mathrm{SF}(\cs)\sim\mathrm{SF}(\cs')$, namely, the
Jordan--H\"older theorem holds true for $A$.
\end{proof}

Since the derived Jordan--H\"older theorem holds true for derived simple algebras, we have
\begin{corollary}
Let $A$ be a $k$-algebra with a block decomposition
$A=A_1\oplus\ldots\oplus A_s$. If $A_1,\ldots,A_s$ are derived simple, then the derived Jordan--H\"older theorem holds true for $A$.
\end{corollary}

Since noetherian rings always admit a block decomposition,
Theorem~\ref{thm:derivedsimplicity-commutativering} yields the following.

\begin{cor} The derived Jordan--H\"older theorem holds true for commutative noetherian rings.
\end{cor}

Let $k$ be a field. Next we restrict our attention to finite-dimensional algebras over $k$. For these algebras the finiteness of any stratification is an easy corollary of  Propostion~\ref{p:rank-and-recollement}. However, as we will show later, the uniqueness property in general fails for finite-dimensional algebras.

\begin{corollary} Let $k$ be a field and
let $A$ be a finite-dimensional $k$-algebra. Then any stratification of $\cd(\Mod A)$ is finite.
\end{corollary}

We give some more examples for which the derived Jordan--H\"older theorem is valid. All these algebras have
two non-isomorphic simple modules, so it follows from Proposition~\ref{p:rank-and-recollement} that the outer algebras of any non-trivial recollement are local and hence are derived simple. Therefore any non-trivial recollement is already a stratification.

\begin{ex}{\rm For the algebra in Example~\ref{e:three-recollements} or Example~\ref{e:three-recollements-and-derived-simple}, there is a unique ladder. This
 in particular verifies the derived
Jordan--H\"older theorem for these two algebras.

For the algebra in Example~\ref{ex:Dminussimple-neq-Dsimple}, up to equivalence there are precisely two non-trivial recollements, which have the same factors. In particular, the derived Jordan--H\"older theorem
holds true for this algebra.}
\end{ex}

The validity of the derived Jordan--H\"older theorem was proposed as an open question by Hong\-xing Chen and Changchang Xi. They informed us that \cite[Theorem 1.1]{ChenXi12} yields further examples of finite-dimensional algebras for which the derived Jordan--H\"older theorem fails. Here we present an example of a different flavour, which is generalised in \cite{LiuYang16}.

\begin{ex}\label{p:counterexampleJH}
{\rm
Let $A$ be the $k$-algebra given by the following quiver with relations
\[\xymatrix{1\ar@(ul,dl)_{\alpha}\ar@<.7ex>[r]^{\gamma} & 2\ar@<.7ex>[l]^{\beta}}, \qquad\beta\gamma\beta,\alpha^2,\gamma\alpha.\]
 We are going to show that the derived Jordan--H\"older theorem does not hold for $A$.

The composition series of the two indecomposable right projective
$A$-modules $P_1=e_1A$ and $P_2=e_2A$ are depicted as follows
\[P_1=\begin{smallmatrix} & 1 & \\[3pt] 1 & & 2\\[3pt] 2 && 1\\[3pt]1&& \end{smallmatrix},
\qquad P_2=\begin{smallmatrix} 2\\[3pt] 1\\[3pt] 2 \\[3pt] 1\end{smallmatrix} . \]

We list, without proof, some properties of the algebra $A$. It has infinite
global dimension. Its finitistic dimension is finite; this follows from the
recollements below, since local algebras have finite (in fact, zero)
finitistic dimension, and therefore the middle term has so, too, thanks to \cite[Theorem 2]{Happel93}. $A$ is a
monomial algebra of wild representation type (see table W in \cite{Han02}).
Its opposite algebra is a standardly stratified algebra (in the sense of
\cite{ADL}; that is, its regular
representation has a filtration by standard modules, using
the order $1<2$).

Consider $e=e_1$. We have isomorphisms $A/Ae_1A=k\{e_2\}\cong k$ and
$e_1Ae_1=k\{e_1,\alpha,\beta\gamma,\alpha\beta\gamma\}\cong k\langle
x,y\rangle/(x^2,y^2,xy)$ of algebras. The algebra $e_1Ae_1$ is local and its regular module has a simple socle; thus the algebra is self-injective.
As a right $A$-module $A/Ae_1A$ is
isomorphic to the simple module at the vertex $2$ and admits the
following projective resolution
\[\xymatrix@C=1.2pc{\ldots\ar[r]&P_1\ar[r]^{\alpha}&P_1\ar[r]^{\alpha}&P_1\ar[r]^{\alpha}&P_1\ar[r]^{\gamma}&P_2\ar[r]&A/Ae_1A\ar[r]&0}.\]
It follows from Lemma~\ref{l:criterion-stratifying-ideal} that $Ae_1A$ is a stratifying ideal. Consequently,  $\cd(\Mod A)$ admits a recollement by
$\cd(\Mod A/Ae_1A)$ and $\cd(\Mod e_1Ae_1)$, i.e. by $\cd(\Mod k)$ and $\cd(\Mod k\langle
x,y\rangle/(x^2,y^2,xy))$.

The case for $e=e_2$ is similar. We have algebra isomorphisms
$A/Ae_2A=k\{e_1,\alpha\}\cong k[x]/(x^2)$ and $e_2Ae_2=k
\{e_2,\gamma\beta\}\cong k[x]/(x^2)$. As a right $A$-module
$A/Ae_2A$ has composition series
$\begin{smallmatrix} 1\\ 1\end{smallmatrix}$ and admits the
following projective resolution
\[\xymatrix{\ldots\ar[r]&P_2\oplus P_2\ar[r]^{\left(\begin{smallmatrix}\gamma\beta & \\ & \gamma\beta\end{smallmatrix}\right)}&P_2\oplus P_2\ar[r]^{\left(\begin{smallmatrix}\gamma\beta & \\ & \gamma\beta\end{smallmatrix}\right)}
&P_2\oplus
P_2\ar[r]^(0.6){(\beta,\alpha\beta)}&P_1\ar[r]&A/Ae_2A\ar[r]&0}.\]
It follows from Lemma~\ref{l:criterion-stratifying-ideal} that  $Ae_2A$ is a stratifying ideal. Consequently, $\cd(\Mod A)$ admits a recollement by
$\cd(\Mod A/Ae_2A)$ and $\cd(\Mod e_2Ae_2)$, i.e. by $\cd(\Mod k[x]/(x^2))$ and
$\cd(\Mod k[x]/(x^2))$.

{\em To summarise: \\
The category $\cd(\Mod A)$ admits two recollements: one  has  factors  $k$
and $k\langle x,y\rangle/(x^2,y^2,xy)$, while the other one has  factors $k[x]/(x^2)$ and $k[x]/(x^2)$. All these factors  are
local algebras, hence derived simple, and clearly pairwise not Morita (and thus also not derived) equivalent. This shows that the derived Jordan--H\"older theorem fails for $A$.}

In both recollements, the functor $i_{\ast}$ sends the respective algebra $B$ on the
left hand side to an $A$-module of infinite global dimension. Hence all criteria
given in section four fail and neither recollement
restricts to $K^b(\proj)$, $D^b(\mod)$ or $D^-(\Mod)$. In the second recollement the functor $j_!= ? \lten_{e_2Ae_2} e_2A$ does restrict to $D^b(\mod)$, since $e_2A$ is free of rank two over its endomorphism ring $e_2Ae_2$. Therefore, by Proposition
\ref{p:ladder-extension}(b), the second recollement can be
extended upwards and thus it is part of a non-trivial ladder of height $\geq 2$.
The first recollement, associated with $e_1$, cannot be extended upwards.
Neither the first nor the second recollement can be extended downwards, since
the criterion in Proposition~\ref{p:ladder-extension}(a) fails in both cases.

}
\end{ex}

\end{document}